\documentclass{article}
\usepackage[utf8]{inputenc}

\usepackage{amssymb}

\usepackage{graphicx,color}
\usepackage{pgfplots}
\pgfplotsset{compat=1.18} 

\usepackage{import}
\usepackage{amsthm}
\usepackage{amsmath}
\usepackage{float}
\usepackage{caption}
\usepackage{subcaption}
\usepackage{hyperref}
\usepackage{multirow}
\usepackage{rotating}
\usepackage{tikz}
\usepackage{enumitem}

\usepackage{mathtools}

\graphicspath{{images/}}

\theoremstyle{plain}
\newtheorem{theorem}{Theorem}

\newtheorem{corollary}[theorem]{Corollary}
\newtheorem{lemma}[theorem]{Lemma}

\theoremstyle{definition}
\newtheorem{remark}[theorem]{Remark}

\newtheorem{definition}[theorem]{Definition}

\providecommand{\com[2]}{\begin{tt}[#1: #2]\end{tt}}

\def\ba{\begin{array}}
\def\ea{\end{array}}
\def\bi{\begin{itemize}}
\def\ei{\end{itemize}}

\def\mP{\mathbb{P}}

\def\mZ{\mathbb{Z}}

\def\mE{\mathbb{E}}
\def\m1{1}
\def\eps{\varepsilon}

\def\cY{\mathcal{Y}}

\def\cB{\mathcal{B}}

\def\b1{{\mathbf 1}}

\def\nka{n^{\frac{k+2}{k+1}}}
\def\nkb{n^{\frac{k+2}{2k+2}}}
\def\nkmb{n^{-\frac{k+2}{2k+2}}}
\def\nkc{n^{\frac{1}{k+1}}}
\def\nkmc{n^{-\frac{1}{k+1}}}
\def\nkd{n^{\frac{1}{2k+2}}}

\DeclareMathOperator*{\argmin}{arg\,min}

\title{Mixing on the cycle with constant size perturbation}

\author{Shi Feng\thanks{S. Feng is with Cornell University, Department of Mathematics (email: {\tt sf599@cornell.edu}). }%
\and Bal\'azs Gerencs\'er\thanks{B.\ Gerencs\'er is with HUN-REN Alfr\'ed R\'enyi Institute of Mathematics, Budapest, Hungary and E\"otv\"os Lor\'and University, Department of Probability and Statistics, Budapest, Hungary, (email: {\tt gerencser.balazs@renyi.hu}). His research was supported by NRDI (National Research, Development and Innovation Office) grant KKP 137490. }}
\date{\today}

\begin{document}

\maketitle

\abstract{Considering a Markov chain defined on a cycle, near-quadratic improvement of mixing is shown when only a subtle perturbation is introduced to the structure and non-reversible transition probabilities are used.
More precisely, a mixing time of $O(\nka)$ can be achieved by adding $k$ random edges to the cycle, keeping $k$ fixed while $n\to\infty$.
The construction builds upon a biased random walk along the cycle.
}

\section{Introduction}

On the quest for finding and creating Markov chains with efficient mixing, a natural approach is to adjust the non-zero transition probabilities of a reference chain. A further step is to increase the number of allowed transitions, and to analyze what can be achieved. In other words, a graph $G=(V,E)$ of allowed transition is given and possibly augmented. 
Once a transition probability structure $P(\cdot,\cdot)$ is specified with stationary distribution $\pi$, finite-time behavior is considered to assess efficiency through 
the \emph {total variation mixing time}, capturing the time needed to approach $\pi$ within a tolerance $\eps>0$:
\[
    t_{\rm mix}(\eps) = \min\{t: \max_i \|P^t(i,\cdot)-\pi(\cdot)\|_{\rm TV} \le \eps\}.
\]

In this paper we analyze an example on the power of minor augmentation, where the number of new edges is kept constant while the instance grows, combining a cycle graph, a biased random walk as a reference chain, and some additional edges. Still, a near quadratic improvement of the mixing time
can be achieved, from $\Theta(n^2)$ down to any $\Theta(n^{1+c}), c>0$, as specified and stated below, closing the gap towards the inherent bound posed by the diameter.

In more detail, we analyze Markov chains on a minor perturbation of the cycle graph $\mZ_n = \mZ/n\mZ$, depending on a few parameters to be fixed throughout the investigation.
Most importantly, let the number of added edges be a constant $k\in\mZ^+$ while $n$ is to be increased. We randomly uniformly choose a $2k$-tuple among the vertices of the cycle and introduce a uniformly random perfect matching among these vertices ( possibly creating multiple edges). These exceptional points will be called \emph{hubs}, with their added edges among them.

In order to acquire asymptotic bounds on mixing, we also fix the local transition structure for the Markov chain on the graph just introduced. That is, choose $p,q,a>0$ such that $p+q+a\le 1$ and $p>q$. Then at any vertex $p$ is the forward, $q$ is the backward transition probability along the cycle $\mZ_n$ using its natural orientation, $a$ is “across”, for the probability of using the added edge, if any. The remaining amount is used as a loop probability, if any. It is easy to see that the stationary distribution is uniform.

Let us now state our main result determining the exact polynomial rate of the mixing time, extending the previous work \cite{gb:ringmixing_singleedge2023}.

\begin{theorem}
\label{thm:mainthm}
Considering the upper bound: for any $\delta>0$ target exception probability there exist $\gamma(\delta)<\infty$ 
with the following property. For $n$ large enough and any $0<\eps<1/2$, with probability at least $1-\delta$ in terms of the $k$ random extra edges we have
\begin{equation}
    \label{eq:mainthm_bound}
    t_{\rm mix}(\eps) \le \gamma(\delta) n^{\frac{k+2}{k+1}} \log\frac 1 \eps.
\end{equation}
Considering the lower bound: there exist $\gamma^*,\eps^*>0$ 
with the following property. For $n$ large enough asymptotically almost surely (a.a.s.) in terms of the $k$ random extra edges we have
\begin{equation}
    \label{eq:thm_lowerbound}
    t_{\rm mix}(\eps^*) \ge \gamma^* n^{\frac{k+2}{k+1}}.
\end{equation}
Furthermore, no cutoff occurs.
\end{theorem}

\begin{remark}
The statement of Theorem \ref{eq:mainthm_bound} is spelled in a parametric way depending on the exception probability required. It is easy to see that an a.a.s.\ upper bound follows for any rate function $\eta(n)\gg n^{\frac{k+2}{k+1}}$. 
Moreover, simulations in Section \ref{sec:simulation} suggest this to be the sharp bound, i.e., $\gamma(0)\nka$ with a fixed constant $\gamma(0)$ the upper bound likely will not hold a.a.s.
\end{remark}

To put the current result in context, there has been quite some work to find speedup via randomized augmentation, results in this direction are able to identify rapid mixing, in $\Theta(\log n)$ for an instance of size $n$ of the (perturbed) Markov chains \cite{hermon2020universality}, in most cases exhibiting cutoff at the entropic time, in a similar way as it would happen when $G$ was fully random \cite{berestycki2018random, ben2018comparing}.
Decreasing the transition probabilities on the newly added edges causes the novel rapid mixing and cutoff to vanish \cite{baran2023phase} (with the critical quantities depending on the expansion of the original graph). The special case of the cycle, mixing on the Bollobás-Chung model, has been studied earlier \cite[Chapter 6.3]{durrett:rgd}.
To achieve these, a perfect matching is added to the base graph, increasing the degree of all vertices, in order to get the suitable expansion and transience. If the augmentation is smaller, resulting only in an average degree increase of some small $c>0$ (via an Erdős-Rényi graph, to be precise), it can be shown \cite{krivelevich2013smoothed} that the mixing time is $O(\log^2 n)$. Indeed, the appearance of a path of $\Omega(\log n)$ can occur (depending on the base graph) with local diffusive behavior.

In the current paper we demonstrate an example with much less perturbation leading to a
non-trivial speedup effect, using only $\Theta(1)$ new connections. However, it is easy to see that any reversible chain on the cycle with $\Theta(1)$ added edges would need $\Omega(n^2)$ time to mix.

The impressive results cited mostly consider the simple random walk (SRW) once the graph is generated. 
Alongside, often the non-backtracking walk (NBRW), a Markov-chain on edges, is investigated and shown to mix faster, see also \cite{diaconismiclo:markov2spectral2013}. Achieving similar speedup for a walk on vertices can be hoped by allowing non-reversible variants. The advantage has been demonstrated by efficient examples \cite{diaconis:nonrev2000}, even generalized to the developed concept of lifting \cite{chen_and_al:lifting1999}, \cite{apers2021characterizing}.

This closes the loop, for our current setup we include the essential non-reversible modification, which leads to the proven speedup.

The structure of the paper is as follows. 
The proof of Theorem \ref{thm:mainthm} is split into Sections \ref{sec:track}, \ref{sec:spread} and \ref{sec:diffuse}. Numerical results complement the current analysis in Section \ref{sec:simulation}, followed by final remarks in Section \ref{sec:conclusions}.

In the analysis to follow a scaling factor $\rho$ will be used along the principal time horizon $\nka$ to allow for fine-tuning. Several global constants will appear, $\lambda_i,\delta_i,\alpha_i,\mu_i$ in the three sections to come, which will not depend on $n,\rho$, but may implicitly depend on $k$.

\section{Structure of typical tracks}
\label{sec:track}

At the first stage, we want to understand typical values of statistics of the Markov chain capturing how added edges are utilized while disregarding the local diffusive effect.

Let us introduce notations for the graph and the Markov chain. 
Let $\mZ'_n$ consist of the vertices $\{0,1,\ldots,n-1\}$ and the edges of $\mZ_n$ augmented by the (random) extra edges, and $P$ be the transition probability matrix as described before driving the Markov chain $X_t, t=0,1,\ldots$
Let us denote by $h_j, j=1,\ldots,2k$ the random hubs on the cycle in a sorted order, $0\le h_1<h_2<\ldots$. 
Let $a_j, j=1,\ldots,2k$ be the lengths of the arcs on the cycle as it is split by the hubs with $a_j = h_{j+1}-h_j$, which is interpreted in $\mZ^+$. Note that this dual viewpoint will be appearing multiple times as we both handle a cycle and its perturbation (naturally on $\mZ_n$) and also distances traveled and similar quantities (naturally on $\mZ$).
The additional edges form a perfect matching among the hubs, to which we introduce an auxiliary orientation, this way we get the directed random perfect matching $e_i=(h_{m^-(i)}, h_{m^+(i)}), i=1,\ldots, k$, the hubs with indices $m^-(i)$, $m^+(i)$ getting connected. The directed matching does not change the transition probability structure originally described, but will allow statistics to be introduced to be well-defined. Also corresponding to the matching we define the pairing function $m(\cdot)$, simply by setting $m(m^-(i))=m^+(i)$ and vice versa.
Let $l_i, i=1,\ldots, k$ be the signed length of $e_i$ in terms of the cycle, taking values in $(-\frac n 2,\frac n 2]\subset \mZ$. A simple illustration can be seen in Figure \ref{fig:demolength}.
\begin{figure}[H]
  \centering
  \begin{subfigure}[b]{0.45\textwidth}
    \centering
      \includegraphics[width=0.7\textwidth]{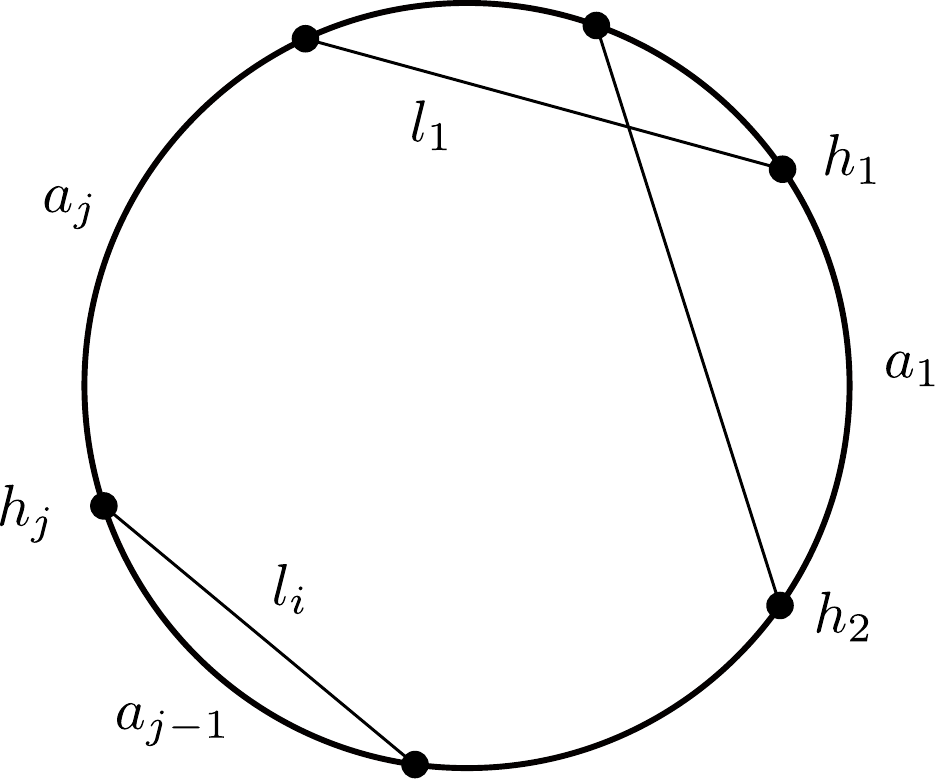}
        \caption{Length type variables.}
  \label{fig:demolength}
  \end{subfigure}
  \hfill
  \begin{subfigure}[b]{0.45\textwidth}
  \centering
      \includegraphics[width=0.7\textwidth]{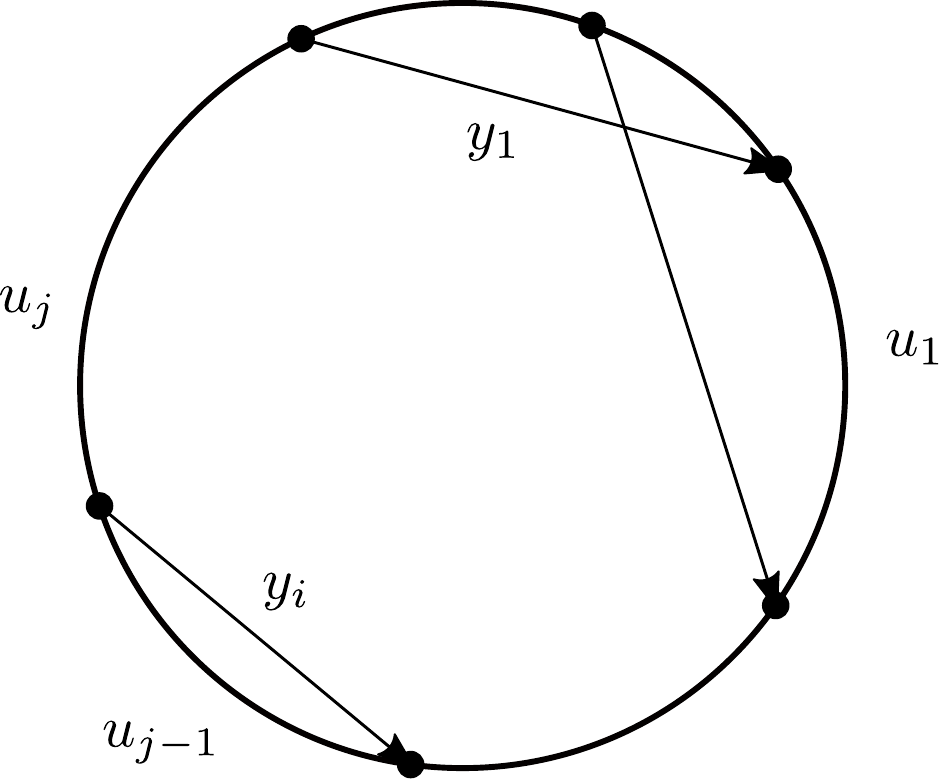}
    \caption{Utilization, local time type variables.}
  \label{fig:demousages}
  \end{subfigure}
  \caption{Illustration of variables as composed in our model.}
  \label{fig:demomodel}
\end{figure}
Let us introduce a scaling constant $0<\rho\leq 1$ that can be adjusted later. Our current goal is to understand the track passed by disregarding the diffusive effects, at a total distance $L=\rho \nka$. This will correspond to time $T=\frac{L}{p-q}$ in terms of the actual Markov chain as presented in Section \ref{sec:diffuse} where diffusive effects are taken into account. For the current purpose the process $\xi_l$ studied is defined as follows.

\begin{definition}
    Consider the universal lift $\tilde \mZ_n'$ of $\mZ_n'$ the lifting $\tilde X_t$ of $X_t$ on it, until a prescribed stopping time $\tau$ is reached. Create the loop erasure $\tilde \xi_l$ from $\tilde X_t$, then project it back to $\mZ_n$ to obtain $\xi_l$.
\end{definition}

In other words, $\xi_l$ can be interpreted as a local loop-erasure, small detours, loops are cancelled, but when an arc is fully passed between two hubs, it is kept. Specifying further stopping will happen when $L$ distance has been travelled. More precisely, for any $0\le t < t'$ we compose $\beta(t,t') = \sum_{s=t}^{t'-1} \beta_0(X_s,X_{s+1}) \in \mZ$, where $\beta_0(X_s,X_{s+1}) = X_{s+1}-X_s \in \{-1,0,1\}\subset \mZ$ when stepping along the cycle and for jumps we set $\beta_0(h_j,h_{m(j)})=0$ (similarly on the lift). Let $\tau = \min\{t:\beta(0,t)\geq L\}$.

The analysis will be cleaner if all the arcs included are passed in the forward direction, let us make sure this is the typical case. 
\begin{lemma}
\label{lm:preferableevents}
Consider the following preferable events:
\begin{align*}
\cB_1 &= \{a_i>\log^2 n,~ i=1,\ldots,2k\}\quad \textrm{--- concerning the graph},\\
\cB_2 &=\{\beta(t,s) > -\log^2 n,~ 0\le t<s\le n^2\}\quad \textrm{--- concerning the walk}.
\end{align*}
For these we have $\mP(\cB_1)\geq 1-\lambda_1 \log^2 n/n$ and $\mP(\cB_2)\geq 1-\exp(-\lambda_2\log^2 n)$ for appropriate $\lambda_1,\lambda_2>0$ as $n\to\infty$.
In particular, no full arc is backtracked once fully passed and thus $\xi_l, l=0,1,\ldots, L$ asymptotically almost surely (a.a.s.) properly represents the trajectory followed.
\end{lemma}
\begin{proof}
    The distance between any two hubs is uniform, so by a union bound on the pairs we get $\mP(\cB^c_1)\leq \frac 12 (2k)^2 2\log^2n/n$.

    Thanks to the drift present, the event of backtracking can be viewed as a standard gambler's ruin situation with bias, providing an exponentially decaying probability in the distance. Thus starting from from any given time there is an $\exp(-\lambda_2\log^2 n)$ probability for a backtrack of $\log^2 n$ length for some $\lambda_2>0$. Taking a union bound on the generous $n^2$ time horizon taken to recover $\cB_2$ and slightly decreasing the constant $\lambda_2$ we arrive at the claim. 
\end{proof}
We want to understand the evolution of $\xi_l$ by itself. Whenever a hub $h_j$ is reached, two things can happen: either with probability $p_G^{(j)}$ eventually it goes(G) on along the cycle hitting the following hub $h_{j+1}$, or with probability $p_J^{(j)}$ eventually it jumps(J) and passes the arc on the other side of the extra edge (in the positive direction), hitting $h_{m(j)+1}$. In principle there is the third possibility of backtracking which is negligible on $\cB_1$ as seen above. We can approximate these probabilities based on the parameters $p,q,a$.
\begin{lemma}
\label{lm:pg_pj}
  On $\cB_1$ for the eventual probabilities to continue at hub $h_j$ we have
  \begin{align*}
  p_G^{(j)} &= p_G+O(\exp(-\lambda_2 \log^2 n)),  &p_G = \frac{p-q+a}{p-q+2a},\\
  p_J^{(j)} &= p_J+O(\exp(-\lambda_2 \log^2 n)), &p_J = \frac{a}{p-q+2a}.
\end{align*}
\end{lemma}
\begin{proof}
  We consider $(h_j, h_{j'})$, $j'=m(j)$, a single pair of connected hubs and the vertices $b,b'$ one step ahead of them, respectively. We simplify calculations by assuming that both hubs are on a copy of $\mZ$ and calculate $p_G,p_J$ for this case. This will provide a good initial approximation thanks to $\cB_1$. The following transitions can occur:
  \begin{itemize}
  \item From $b$, eventually exiting to $+\infty$ with probability $(p-q)/p$ or returning to $h_j$ with probability $q/p$. Analogously for $b'$.
  \item From $h_{j}$, stepping to $b$ with probability $p$, to $h_{j'}$ with probability $a$. All other choices (including stepping backwards) simply cause looping, thus in total there is an eventual move with probability $p/(p+a)$ and $a/(p+a)$ to the two neighbors, respectively. Analogously for $h_{j'}$.
  \end{itemize}
  
  With reference to $h_j$ we get the event corresponding to $p_G$ if we step to $b$ and then exit to $+\infty$. Before we are allowed to take any number of detours to $b$ and back. When we happen to step to $h_{j'}$, we need the jumping event from that reference which allows to formulate the defining equations
\begin{align*}
  p_G &= \sum_{j=0}^\infty \left(\frac{p}{p+a} \frac q p\right)^j\left(\frac{p}{p+a}\frac{p-q}{p}+\frac{a}{p+a} p_J  \right),\\
  1 &= p_G+p_J.
\end{align*}
Combining and simplifying it reads:
$$
p_G = \frac{1}{1-\frac q {p+a}}\left(\frac{p-q}{p+a} + \frac{a}{p+a}(1-p_G) \right)
=\frac{p-q}{p-q+a} + \frac{a - a p_G}{p-q+a}.
$$
From here we get
$$
p_G = \frac{p-q+a}{p-q+2a}, \qquad p_J = \frac{a}{p-q+2a}.
$$
It remains to transfer the computation from the idealized double $\mZ$ to the finite arcs connected to $h_j,h_{j'}$. Thanks to $\cB_1$ these arcs have to be sufficiently long, so that they introduce only minor changes in the dynamics:
\begin{itemize}
    \item Exiting to $+\infty$ is modified to arriving to the next hub, thus removing a small positive probability of returning.
    \item Looping in the negative direction might include returning to the preceding hub with a small positive probability.
\end{itemize}
In both cases a change happens only if a full arc is backtracked, which are all of length at least $\log^2 n$, thus contributing with an error based on the comparison of hitting times $\mP(\tau_{-\log^2 n}<\tau_{+\infty}) \leq \exp(-\lambda_2 \log^2 n)$, confirming the claim.
\end{proof}

In contrast to the general step above one has to be careful near $\xi_0$ and $\xi_L$, but only a constant factor correction will appear.

Having seen how $\xi_l$ can evolve when arriving at a hub, we now want to get an overall picture. This is done by defining total counts of various decisions, then combined with Lemma \ref{lm:pg_pj} above this will allow us to identify a typical set of trajectories.

Let us denote by $N_{G,j}(L)$ the count of eventual $G$'s at $h_j$ up to distance $L$ (omitted when clear from the context), similarly for $N_{J,j}(L)$. This allows us to define
\begin{align*}
y_i &= N_{J,m^-(i)} - N_{J,m^+(i)}, \quad i=1,\ldots, k\\
u_j &= N_{G,j} + N_{J,m(j)}, \quad j=1,\ldots, 2k
\end{align*}
Observe that $y_i$ counts the signed total traffic through an extra edge while $u_j$ counts the instances of the arc $[h_j,h_{j+1})$ having passed through. See Figure \ref{fig:demousages} for illustration. The latter concept will be refined in a second to account for the process near $\xi_0,\xi_L$. The quantities are also interrelated by 
\begin{equation}
    \label{eq:y_to_u}
    \begin{aligned}
u_{j} &= u_{j-1} - y_i \qquad  \text{if }m^-(i)=h_j,\\
u_{j} &= u_{j-1} + y_i \qquad \text{if }m^+(i)=h_j.
    \end{aligned}
\end{equation}

We are interested in a meaningful lower bound for the probability of any
\[
y := (y_1,\ldots,y_k)\in Y := [-\lambda_3\sqrt{\rho}\nkd,\lambda_3\sqrt{\rho}\nkd]^k,
\]
to be realized for appropriate global constant $\lambda_3>0$. First we derive concentration relations in terms of $u:=(u_1,\ldots,u_{2k})$.

Let us fix the starting position $X_0=\xi_0$ of the Markov chain and thus the track.
Given a vector $y$ the final position can be calculated as
\begin{equation}
  \label{eq:xld_def}
\xi_L = \xi_0 + L + \sum_{i=1}^k y_i l_i \qquad \mod n
\end{equation}
that is, the $L$ step shift is modified by the jumps prescribed by the $y_i$ given.

Given $\xi_0, y$, and thus $\xi_L$, we can now slightly extend our arc usage quantities as there will be partially used arcs when leaving $\xi_0$ and when reaching $\xi_L$. If $\xi_0\in (h_j,h_{j+1})$ we introduce $u_j^-,u_j^+$ for the usage of $[h_j,X_0),[X_0,h_{j+1})$ respectively. Similarly, if for the given $\xi_0, y$ the endpoint $\xi_L$ is falling inside $(h_{j'},h_{j'+1})$, the usages $u_{j'}^-,u_{j'}^+$ are defined for the two parts there. During the upcoming development we will add comments where relevant how minor adjustments are needed for these split intervals.

Observe that in general $u_j^- +1 = u_j^+$ and $u_{j'}^- -1 = u_{j'}^+$, and the appropriate version has to be used in \eqref{eq:y_to_u}. If $\xi_0, \xi_L$ happen to fall on the same arc, it might be split in three parts, with all computations working analogously.

\begin{lemma}
\label{lm:ys_to_us}
Let $y\in Y$, then there exist corresponding integer vector $u\in (\mZ^+)^{2k}$ and $u_j^\pm$ at $\xi_0,\xi_L$ which are consistent with $y$ in terms of \eqref{eq:y_to_u} and are unique.
Also $\left|u_j-\rho\nkc\right|<\lambda_4\sqrt\rho \nkd$ for all $j=1,\ldots,2k$ and some $\lambda_4>0$, with the same error bound for $u_j^\pm$ where defined. 
\end{lemma}
\begin{proof}
We rely on $\sum_{j=1}^{2k^*} a_j u_j = L$, expressing that the coverage counted by arcs add up to the overall length. Here with a slight abuse of notation the summation to $2k^*$ refers to including $u_j^+, u_j^-$ terms if relevant, paired with the corresponding partial arc lengths.

Choose any favorite arc to be a reference, $\tilde u = u_{j_0}$ whose value is yet to be determined. By the condition on $y_i$ and iterating \eqref{eq:y_to_u} we get for all $j$
$$
u_{j} = \tilde u + O(\sqrt{\rho}\nkd),
$$
including $u_{j}^\pm$ where relevant and being consistent when arriving back as each $y_i$ is included both with positive and negative sign. We stress that we use the $O(\cdot)$ notation such that only constants independent of $n,\rho$ are allowed. Combining all the arcs leads to
$$
L = \sum_{j=1}^{2k^*} a_j u_j = n\tilde u + O(\sqrt{\rho} n^{1+\frac{1}{2k+2}}).
$$
Here based on \eqref{eq:xld_def} the difference of the two sides is divisible by $n$, in particular $\tilde u\in\mZ^+$.
Recalling the choice of $L = \rho\nka$ we arrive at
$$
\tilde u = \rho\nkc + O(\sqrt{\rho}\nkd),
$$
in line with the claim. Iterating \eqref{eq:y_to_u} again provides the conclusion for all $u_j$ (and $u_j^\pm$, where relevant).
\end{proof}

Now we are ready for bounding the probability of attaining the vectors $y$ in the concentrated region $Y$.
\begin{lemma}
    \label{lm:ytargetprob}
  Assume $\xi_0$ to be fixed and $\cB_1$ to hold. For any target $\hat y := (\hat y_1,\ldots,\hat y_k)\in Y$, the probability of the realization of this sequence can be bounded from below, i.e.\ for some constant $\lambda_6>0$
    $$\mP\big(y = \hat y \big) \ge \lambda_6 \rho^{-k/2}n^{-\frac{k}{2k+2}}.$$
\end{lemma}
\begin{proof}
At first let us confirm corner cases do not cause a disruption in the analysis. The event of an arc being backtracked, in the spirit of (a looser form of) $\cB_2$ has subpolynomial probability. From $\xi_0\in (h_j,h_{j+1})$ for convenience we prefer to hit $h_{j+1}$ which has probability bounded away from 0 thanks to the drift. At $\xi_L$ the probabilities of $G,J$ might be different thanks to the early halt, but both are bounded away from 0. This means an extra multiplicative factor is sufficient to correct for these cases, we just simply need to adjust $\lambda_6$ at the end.

Now for the overall track, we bound the probability using a new realization of the same distribution.
At each hub $h_j$, install an infinite source $B_{j,\cdot}$ of i.i.d.\ $0-1$ bits with probability $p^{(j)}_G,p^{(j)}_J$, corresponding to $G$ (go ahead), $J$ (jump) as before.

To see the correspondence, start a track from $\xi_0$ in the positive direction, until hitting some $h_j$. Read the first bit there, i.e. $B_{j,1}$ and continue along the cycle in place (if 0 for $G$) or jump across and continue there (if 1 for $J$) depending on the bit. Carry on this process, always reading the next new bit at the hub when arrived, until length $L$ is reached. In fact, this can be viewed as a natural interlacement of the bitstreams. This representation also shows that the number of queries at $h_j$ is determined by the usage of the preceding arc leading to it, $[h_{j-1},h_j)$. Without the interlacement, they form independent sequences, which is the property to be exploited.
For the overall counts let us introduce
\begin{align*}
\tilde N_{J,j}(d) &= \sum_{s=1}^d B_{j,s}, &
\tilde N_{G,j}(d) &= d- \sum_{s=1}^d B_{j,s},
\end{align*}
for the number of $J$ and $G$ bits respectively at $h_j$ during the first $d$ bits.

From the specified $\hat y$ and $\xi_0$ we can compute $\hat \xi_L$ and the corresponding $\hat u = (\hat u_1,\ldots,\hat u_{2k})$ by the means of \eqref{eq:xld_def} and Lemma \ref{lm:ys_to_us}.
Using these, for all $i,j$ define the following (parametric) events that capture the coupling to the random track:
\begin{align}
    A_i(\hat r_i) &= \{\tilde N_{J,m^-(i)}(\hat u_{m^-(i)-1}) = \hat r_i+\hat y_i\}
         \cap \{\tilde N_{J,m^+(i)}(\hat u_{m^+(i)-1}) = \hat r_i\}. \label{eq:Asets}
\end{align}
Here $A_i(\hat r_i)$ expresses a possibility to get $\hat y_i$ signed total usage on the edge $e_i$ until checking the bitstreams until different, but prescribed horizons associated to the predicted readings. As such, for the variables in their definition, use $\hat u^+_j$ when available.

Assume the event $\bigcap_{i=1}^k A_i(\hat r_i)$ and consider the realisation of $\xi_l$ following the bits given by $B_{j,\cdot}$ as described above until reaching $\xi_L$.

The potential issue could be that the number of bits read by the walk does not match the horizon to which $A_i(\hat r_i)$ was defined a priori. We will show this cannot happen, arguing by contradiction.

If there were streams read more than prescribed while following $\xi_l$, let us focus on the very first moment a bit would be read further than planned, say at hub $h_j$. The number of bit readouts is determined by the usage of the arc before, $[h_{j-1},h_j)$, which thus must have already passed $\hat u_{j-1}+1$ to trigger this unfortunate situation.

How could this usage be built up? Partially from $h_{j-1}$ and going on and from $h_{m(j-1)}$ and jumping. But even if we read all bits that were planned there ($\hat u_{j-2}$, $\hat u_{m(j-1)-1}$ respectively), it will provide only a total of $\hat u_{j-1}$ usage, thanks to the connection \eqref{eq:y_to_u} and the respective event $A_i(\hat r_i)$. Consequently for that extra 1 pass of the arc, we must have already read further than the horizon at $h_{j-1}$ or $h_{m(j-1)}$, contradicting that we are about to reveal the very first such bit.

Note that $h_j$ and $h_{m(j-1)}$ might coincide but it doesn't affect the argument. Near $\xi_0$, the same logic works with $u_j^\pm$ while the splitting at $\xi_L$ should not be introduced until having arrived there.

Combining with the fact that $L = \sum_{j=1}^{2k^*} a_j \hat u_j = \sum_{j=1}^{2k^*} a_j u_j$ for the prescribed and the realized usages while $\hat u_j \le u_j$ from the previous argument, we must have $\hat u_j = u_j$ for all $j$, including for the split arcs.

We now turn to estimate the probability of
$\bigcap_{i=1}^{k} A_i(\hat r_i)$ having confirmed that it matches the event of fully realizing $\hat y$ together with an internal parameter vector $\hat r=(\hat r_1,\ldots,\hat r_k)$. We will do the estimation for a range of $\hat r$, each chosen from %
\begin{equation}
  \label{eq:rrange}
R := [p_J\rho\nkc-\lambda_5\sqrt{\rho}\nkd,p_J \rho\nkc+\lambda_5\sqrt{\rho}\nkd]^k,
\end{equation}
with arbitrary $\lambda_5>0$.

In the current construction however, all $A_i(r_i)$ are independent, and by the above claim it is sufficient to bound their probability. They are all the products of two independent events on Binomial distributions. 
Using the choice of $\hat r_i,\hat y_i$, the bound of Lemma \ref{lm:ys_to_us} on $u_{m^-(i)}=\rho\nkc+\sqrt\rho O(\nkd)$ together with Lemma \ref{lm:pg_pj} for the Binomial probabilities we can write
$$
|p_J^{(m^-(i)-1)} \hat u_{m^-(i)-1}-(\hat r+\hat y_i)|\le \lambda_4\sqrt\rho\nkd + \lambda_2' \rho \nkc e^{-\lambda_2\log^2n}
\le \lambda_4'\sqrt\rho\nkd
$$
for $n$ large enough, and similarly around $u_{m^+(i)-1}$.
Stressing that $p^{(j)}_J$ are global constants up to subpolynomial error, a crude Local Central Limit approximation allows us to bound
\begin{equation}
  \label{eq:rbound_single}
  \begin{aligned}
    \mP(A_i(\hat r_i))&\ge \left(\frac {\lambda_5'} {\sqrt{\rho}\nkd}\right)^2= \lambda_5'^2\rho^{-1}\nkmc,\\
    \mP\left(\bigcap_{i=1}^{k} A_i(\hat r_i)\right)&\ge \lambda_5'' \rho^{-k} n^{-\frac k {k+1}}.
  \end{aligned}
\end{equation}
In order to get a lower bound primarily for $\hat y$, we can sum up for the range of $\hat r\in R$, which correspond to disjoint events. In the end we get
$$
\sum_{\hat r \in R} \mP\left(\bigcap_{i=1}^{k} A_i(\hat r_i)\right) \ge (2\lambda_5\sqrt{\rho}\nkd)^k \lambda_5'' \rho^{-k}n^{-\frac k {k+1}}=: \lambda_6\rho^{-k/2}n^{-\frac k {2k+2}},
$$
confirming the claim of the lemma.
\end{proof}

The bound obtained is exactly of the order sought after, both in terms of $n$ and also in the parameter $\rho$, as it inversely matches the number of $\hat y$ available up to constants.

\begin{remark}
    While generating the track using the bitstreams note that the extra $-1$ at $\xi_L$ is the key to freeze to process at the appropriate $\hat u$. If we had removed that, we would have a first extra bit read at $h_{j'+1}$ if $\xi_L\in [h_{j'},h_{j'+1})$, from where all $u$ and all bitstream reads would obviously diverge.
\end{remark}

\section{Spread of typical tracks}
\label{sec:spread}
Denote by $\xi_{L,y}$ be the position of $\xi_L$ for given $y = (y_1,\ldots,y_k)$. Fix $L$ as before and $\xi_0=0$ and consider the collection of points $\Xi_m = \{\xi_{L, (y_1,\ldots,y_k)} \;|\; 0 \leq y_i \leq m , 1\leq i \leq k\}$ for some $0<m<\infty$.  Lemma \ref{lm:ytargetprob} implies that $P(\xi_L\in \Xi_{m}) >0$ and $\Xi_m$ has the good property that the realization of any point inside is bounded in the same order for any $m = \Theta(\sqrt{\rho}\nkd)$. In line with the previous section, instead of focusing on every possible position of $\xi_L$, we can limit our attention to $\Xi_m$. If for some $m = \Theta(\sqrt{\rho}\nkd)$ we show that $\Xi_m$ is spread out evenly on the cycle $\mZ_n$ then we can reasonably conclude that the range of $\xi_L$ cannot be very clustered either. From now on, unless specified we consider $m$ to be of this order.

Here we set up some notations. Define $Y^+_m = \{(y_1,...,y_k)\;|\; 0\leq y_i < m, y_i\in \mathbb{Z}\}$. Define a set $\mathcal{Y}_m$ that contains all coordinate-wise sign flip of $Y_m^+$. Formally we denote $b = (b_1,\ldots,b_k)\in\{-1,1\}^k$ a sign pattern, then write
\begin{align*}
    Y^b_m &= \{(b_1y_1,\ldots,b_ky_k) \mid 0\leq y_i < m, y_i\in \mZ\},\\
    \mathcal{Y}_m &= \{Y^b_m \mid b\in \{-1,1\}^k\}.
\end{align*}
To better understand the structure of this set and what follows, we illustrate in Figure \ref{fig:Ymb} what $Y_m^b\in \cY_m$ represent when $k=3, b=(b_1,b_2,b_3) = (-1,1,-1)$.
\begin{figure}[hbt!]
  \centering
  \includegraphics[scale=0.15]{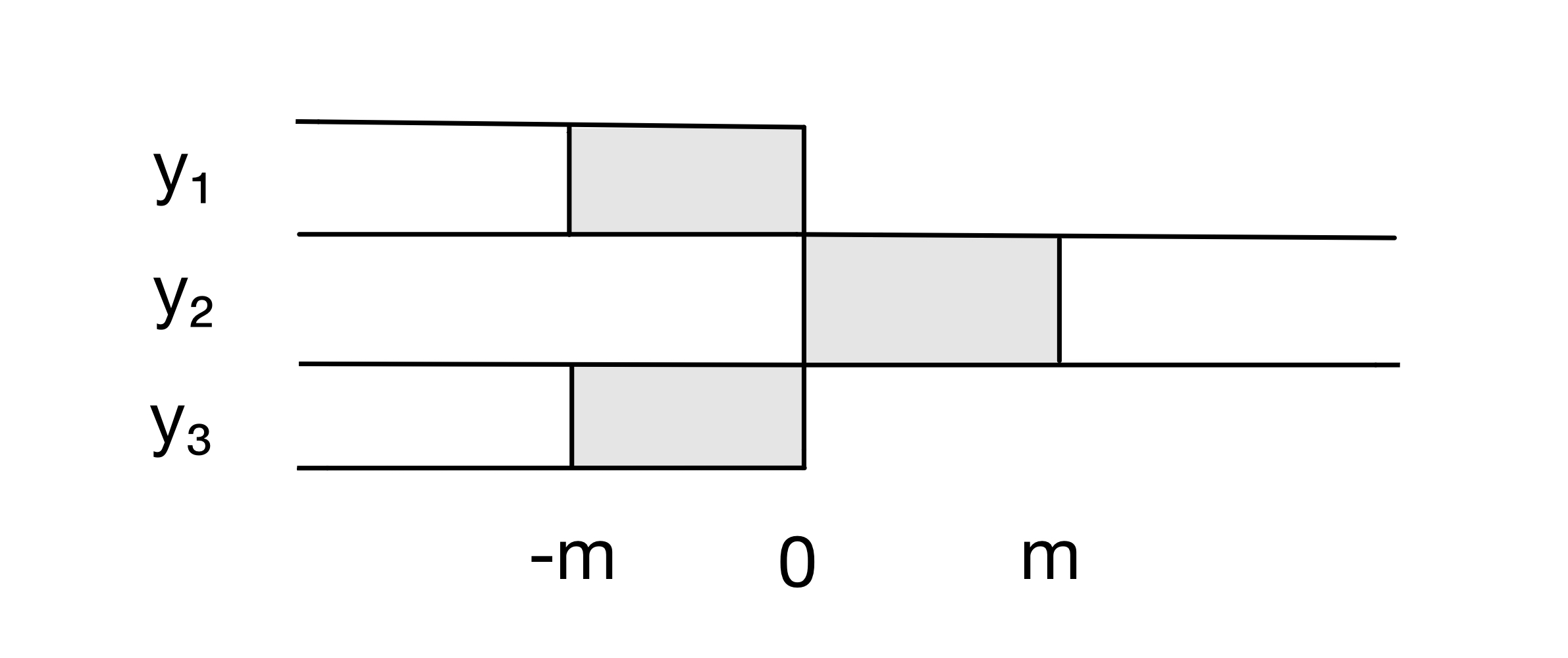}
  \caption{The shaded areas are the possible values that $y_1,y_2,y_3$ can take for $Y_m^{(-1,1,-1)}$}
  \label{fig:Ymb}
\end{figure}
Recall that $l = (l_1,\ldots,l_k)$ stands for the length of the added edges. Given $l$, define the function $f_l: \mZ^{k} \to \mZ_n = \{0,1,\ldots,n-1\}$ as 
\begin{align*}
    f_l\big((y_1,\ldots,y_k)\big) = \sum_{i=1}^{k} y_il_i \quad (mod\; n)
\end{align*}
We may naturally extend $f_l$ also to subsets of $\mZ^k$ by elementwise application.
Note that by construction 
$    \xi_{L,y} = f_l(y) + L ~(mod\; n).$
Therefore, our strengthened goal for now is to prove that $f_l(Y_m^+)$ is spread out evenly on $\mZ_n$. 
We use the original graph distance of $\mZ_n$ to quantify this, i.e. let the distance for $0\leq x\leq x' < n$ to be $d(x',x)=d(x,x') = \min(x'-x, n-x'+x)$.
Note that the cardinality of $Y_m^+$ is $|Y_m^+| = \Theta(m^k)$. Thus we would expect that for every point in $f_l(Y_m^+)$, there should be close-by points on both sides with distances in the order $\Theta(\frac{n}{m^k})$ for it to spread out evenly. Having this target we denote 
 \begin{align}
 \label{eq:standarddistance}
    s = \frac{n}{m^k} 
 \end{align}
 as the "standard distance". One important observation is that for $y,y'\in \mZ^k$, 
 \begin{align}
     d(f_l(y),f_l(y')) = d(f_l(y-y'),0)
 \end{align}
 by linearity of $f_l$. We will use this property in the following proofs without further explanation. In the following lemmas, we consider $m = \Theta(\sqrt{\rho}\nkd)$, and exploit the randomness of $l$. Observe that it can be coupled with i.i.d.\ uniform coordinates with $o(1)$ exception probability. This is present due to hubs being single-use, but can be disregarded anyway in comparison with the constant exception probabilities we are about to analyze.

\begin{lemma}
\label{not too close lemma} 
For any $0<\delta_1<1$ as exception probability, there exists $\alpha_1>0$ such that with probability at least $1-\delta_1$, for all $Y_m \in \cY_m$ and for all $y,y'\in Y_m$ we have $d(f_l(y),f_l(y'))>\alpha_1 s$. In other word, with probability arbitrarily close to 1, the distance between any two points in any $f(Y_m)$ should be larger than the order of standard distance $s$.
\end{lemma}

\begin{proof}
First suppose for technical simplicity $n$ is a prime number. Using the i.i.d.\ uniform random variables $l_1,\ldots,l_k$ on $\mZ_n$ we have 
$y_il_i$ uniform as well unless $y_i=0$ and clearly also for their convolution $\sum_{i=1}^{k} y_il_i ~(mod\;n)$ as long as there is at least one nonzero $y_i$ coordinate, in other words unless we have $y=(0,\ldots,0)$.

Consider the point set $Y_m' = \{(y_1,\ldots,y_k) \mid -m \leq y_i \leq m)\}$. By the property we found above, $\forall y \in Y'_m \setminus \{(0,0,\ldots,0)\}$, $f_l(y)$ is a uniform random variable in $\mZ_n$. Note that $|Y'_m| = (2m+1)^k$ , so $f(Y'_m\setminus 0)$ is a collection of $(2m+1)^k-1$ dependent uniform random variables. Choose $0<\delta_1'<\delta_1$ and let us identify a target neighborhood of $0$ to avoid as follows: 
$$
A = \left\{x\in \mZ\;|\; d(0,x)\leq \frac{\delta_1'}{2^{k+1}}s\right\}.
$$
Checking for the possible $y$, let $I_i$ be the indicator function of $\{f_l(y(i))\in A\}$ for $y(i) \in Y'_m \setminus \{(0,0,\ldots,0)\}$ as $1\leq i \leq (2m+1)^k-1$.
Therefore, 
\begin{equation}
\label{eq:eofI}
\mE[I_i] \leq \frac 1n\left(\frac{2\delta_1's}{2^{k+1}}+1\right).    
\end{equation}
Let $N_1$ to be the total number of $y(i)$ that fall in $A$. Then $\mE[N_1] = \sum_{i=1}^{(2m+1)^k-1} \mE[I_i] = \delta_1' +o(1)$. Thus, 
\begin{align*}
    \mE[N_1] \geq P(N_1\geq 1) = 1-P(N_1=0) \implies P(N_1=0)\geq 1-\delta_1+o(1).
\end{align*}
Let $\alpha_1 = \frac{\delta_1'}{2^{k+1}}$. Note that if $P(N_1=0)$, $\forall Y_m \in \cY_m$, $\forall y\geq y'\in Y_m$, there always exists $y'' = y-y' \in Y'_m$ such that $d(f(y),f(y')) = d(f(y''),0)>\alpha_1s$. Therefore, with probability at least $1-\delta_1$, $d(f(y),f(y')) = d(f(y''),0) >\alpha_1s$.\\

Now let us extend the proof for general $n$. Our aim is to bound $\mE[I_i]$ in a similar way, from where the rest follows. Given a specific choice of $y\in Y_m'\setminus \{(0,0,\ldots,0)\}$ the sum $\sum_{i=1}^{k} y_il_i (mod\;n)$ becomes a uniform random variable on $g\mZ_n$, i.e. $\mZ_{n/g}$ embedded into $\mZ_n$,
with $g = gcd(y_1,y_2,...,y_k,n)\leq m$ (e.g., by checking their Discrete Fourier Transform). Thus we encounter a similar situation on a smaller cycle on $n'=n/g$ vertices by scaling. The only change caused to $\eqref{eq:eofI}$ is replacing the $+1$ term by $+g$.

However, as long as $\frac{2\delta_1's}{2^{k+1}} \gg g$, this change can be neglected by adjusting $\delta_1'$ (towards the target $\delta_1$). Referring to the the definition of $s$ in \eqref{eq:standarddistance} and the $Y_m'$ considered above, this condition is equivalent to $\nkc \gg m$, which is comfortably satisfied within the framework set at the beginning of the section.
\end{proof}

Let us note that by the same reasoning the distances can be controlled when $y$ has inhomogeneous constraints on its coordinates, stated as follows:
\begin{corollary}
    \label{not too close lemma addon}
    For arbitrary constants $0<\eps_i<1, i=1,\ldots,k$ we may consider $Y_{\eps_1m,\ldots,\eps_km} = \{(y_1,\ldots,y_k) \mid 0\leq b_iy_i < \eps_im\}$. Then with the same $\delta_1,\alpha_1$ as in Lemma \ref{not too close lemma} with probability at least $1-\delta_1$ we have that for all possible sign pattern $b$ and $y,y'\in Y_{\eps_1m,\ldots,\eps_km}$ there holds $d(f_l(y),f_l(y'))>\alpha_1 s/\prod_{i=1}^k \eps_i$.
\end{corollary}

Now we have the property that no two points can be very close to each other, within each of the mapped sets in $\cY_m$. As a next step, we can prove that there is a nonzero mapped point that falls close to 0 on an appropriate scale.

\begin{lemma}
\label{one side lemma} 
For any $ 0<\delta_2<1$ as exception probability, there exists $\alpha_2<\infty$ such that with probability at least $1-\delta_2$, for all $Y_m\in \cY_m$, there exists $0 \neq y \in Y_m$ such that $\alpha_1s <d(f(y),0) \leq \alpha_2 s$. In other word, with probability 1, for all possible $Y_m$, there exists a non-zero point $y$ in $Y_m$ such that the distance between $f(y)$ and $0$ is in the order of the standard distance $s$.
\end{lemma}
\begin{proof}
A lower bound has been proved in the previous lemma for all points with $\alpha_1$. The main ingredient for this proof is kind of pigeonhole principle: for a cycle of length $n$ with selected points on it, if the distance between any two subsequent point is larger than $\alpha_1 s$, the total number of points has to be smaller than $\frac{n}{\alpha_1 s}$. 

Since $m$ can be an arbitrary number scaling with $m = \Theta(\sqrt{\rho}\nkd)$, it is sufficient to give an upper bound for $Y_{xm}$ for some integer $0<x<\infty$. Therefore, our goal is to prove that for any specific $Y_{xm} \in \cY_{xm}$, with probability $1-\delta_2$ there is a point in $f(Y_{xm}\setminus 0)$ that falls near $0$. Then the probability of their intersection(all $Y_{xm} \in \cY_{xm}$) can also be controlled arbitrarily close to $1$ by adjusting $\delta_2$.

To guide the proof, we use $k=3$ and the $Y_m^b$ has a $b$ of $(1,-1,1)$ as an illustration. The following images illustrate this case. But the following proof works for any $k$ and any $b$ in general. We can use Figure \ref{fig:Ym_b_example} to represent the $Y_m^{(1,-1,1)}$.
\begin{figure}[H]
  \centering
  \includegraphics[scale=0.15]{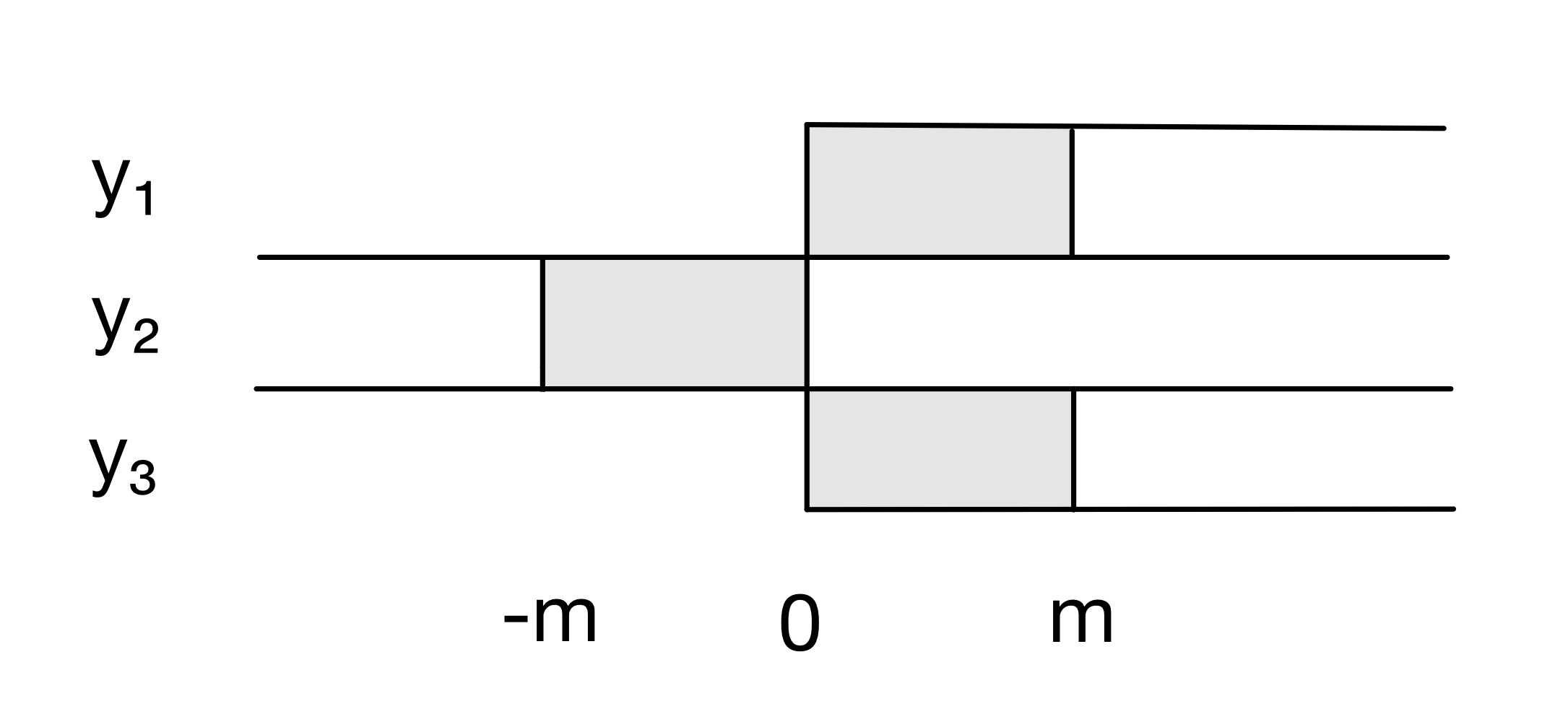}
  \caption{Intuitive illustration of $Y^{(1,-1,1)}_m$}
  \label{fig:Ym_b_example}
\end{figure}
To have a simpler notation, when $b$ is given, we denote 
\begin{align*}
    Y_m(j) = \{(y_1,\ldots,y_k) \;|\; (j-1)m\leq b_iy_i < jm\}.
\end{align*}
An illustration of $Y_m(j)$ can be seen in Figure \ref{fig:Y_mj_example}.
\begin{figure}[H]
  \centering
  \includegraphics[scale=0.15]{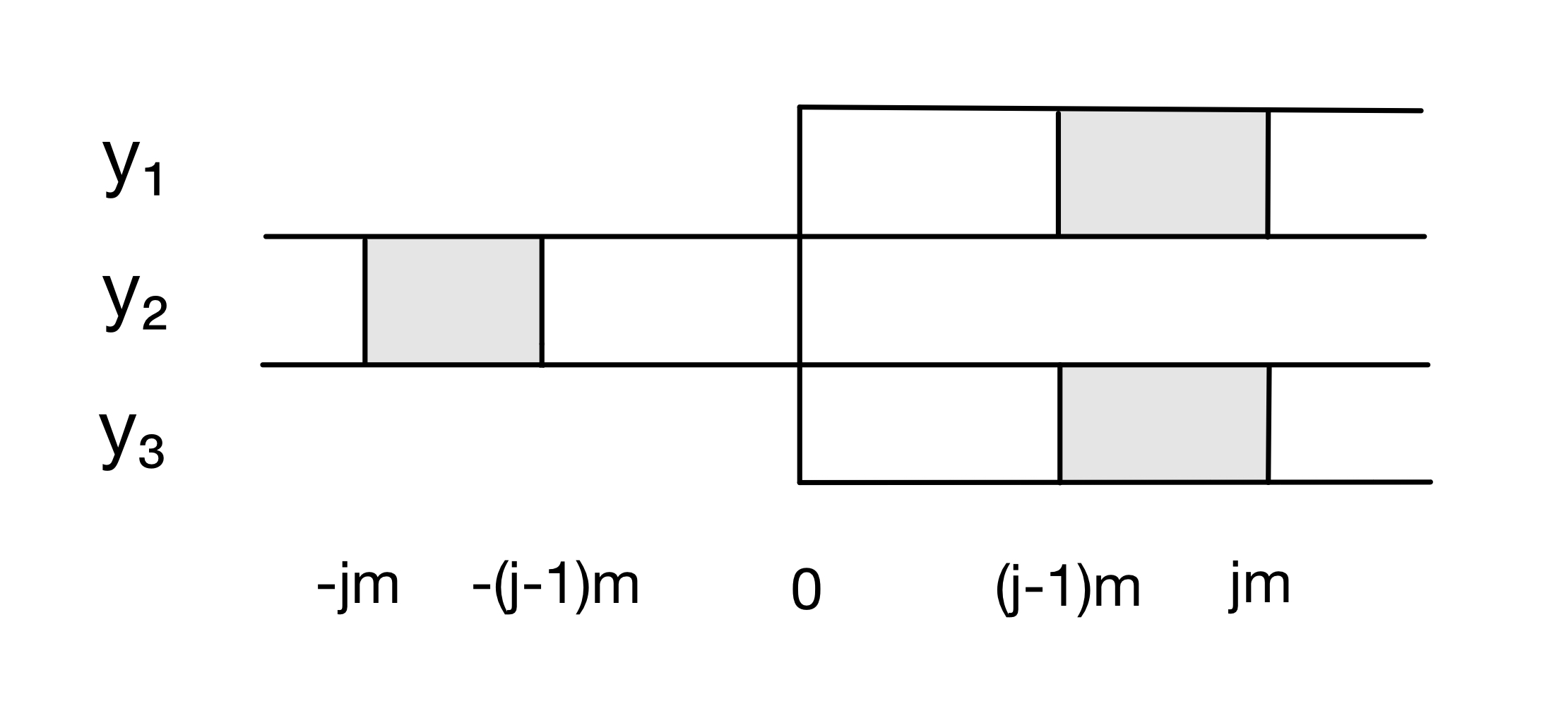}
  \caption{Intuitive illustration of $Y_m(j)$}
  \label{fig:Y_mj_example}
\end{figure}

Observe that $f(Y_m(j))$ is simply a shift of $f(Y_m(1))$. Formally, if we express the addition between a set and real number element-wise to be $\{x_1,\ldots, x_k\} +x = \{x_1+x,\ldots, x_k+x\})$, we have
\begin{align*}
f(Y_m(j)) = f(Y_m(1)) + \sum_{i=1}^{k} b_i \cdot(j-1)\cdot m\cdot l_i\; (mod \; n)   
\end{align*}
Suppose the the distance between two closest points in $f(Y_m(1))$ is larger than $\alpha_1 s$ for $\alpha_1>0$ with the help of Lemma \ref{not too close lemma}. Then the two closest points in $f(Y_m(j))$ is also automatically larger than $\alpha_1 s$ for any $j$ since it is just the shifted version of $Y_m(1)$. 
Now we place a number of $f(Y_m(j))$ to the cycle until it cannot hold that many points if it want to insist that the distance between two closest points must be larger than $\alpha_1 s$. Formally, by a pigeonhole principle we want to choose a finite integer $x$ such that
\begin{align*}
    \left|\bigcup_{j=1}^{x} Y_m(j)\right| > \frac{n}{\alpha_1 s}.
\end{align*}
This is achievable since $|Y_m(j)| = m^k = \frac{n}{s}$ for all $j$ and we can simply choose $x=\lceil\frac{1}{\alpha_1}+1\rceil$. Therefore, within $\bigcup_{j=1}^{x} Y_m(j)$, there must be two points $y'$ and $y''$ such that for their mapped distance $d(f(y'),f(y''))<\alpha_1s$. Note that $y'$ and $y''$ cannot be in the same $Y_m(j)$ since the distance of any two points in a single $Y_m(j)$ is bounded below by $\alpha_1s$. Therefore, $y'$ and $y''$ belong to different $Y_m(j)$, implying $y'-y''$ or $y''-y'$ should have the same sign pattern as $b$ that we started with. See Figure \ref{fig:Ymj-pigeon} for a follow-up illustration. 

\begin{figure}[H]
  \centering
  \includegraphics[scale=0.15]{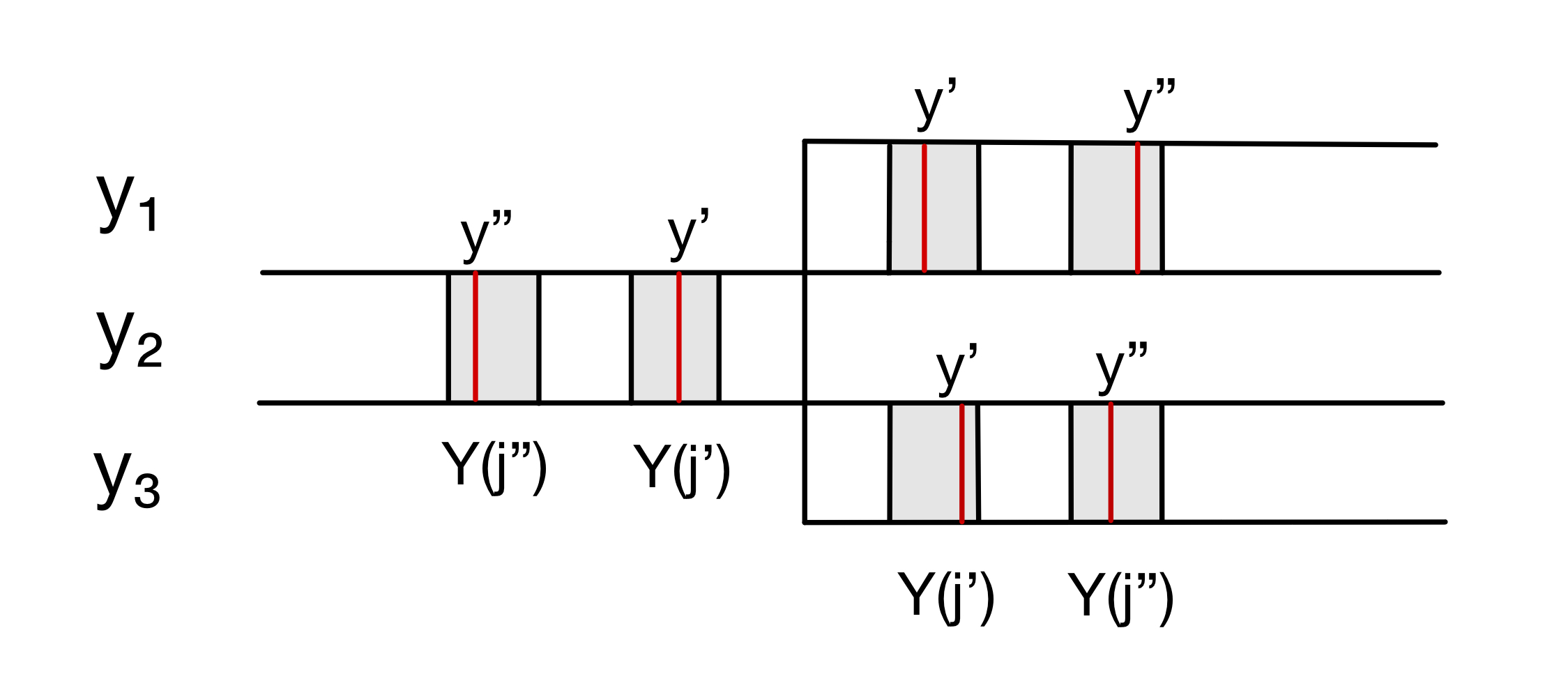}
  \caption{The red lines represents the location of $y'$ and $y''$. $y''-y'$ has the same "sign" indicator as $b$}
  \label{fig:Ymj-pigeon}
\end{figure}
Suppose w.l.o.g.\ that $y''-y'$ has the same sign pattern as $b$. Then $y''-y'$ is in the set $Y_{xm}^b$. Therefore, in $Y_{xm}^b$ there exist $y''-y'$ such that 
\begin{align*}
    d(f(y''-y'),0) = d(f(y'),f(y'')) < \alpha_1 s,
\end{align*}
in line with the claim we are after.
Given we can prove the property for $xm$, we can prove it for $m$ by appropriate scaling, which concludes our proof.
\end{proof}

Again we rely on the viewpoint that there is positive and negative direction on $\mZ_n$, allowing local comparisons despite the lack of global ordering. We extend the above lemma showing that there exist two points near 0 that are on the both side of 0 for each $Y_m$. 

\begin{lemma}
\label{both sides lemma}
For any $ 0<\delta_3<1$ as exception probability, there exists $\alpha_3>0$ such that with probability at least $1-\delta_3$, for all $Y_m \in \cY_m$, there exists $y',y'' \in Y_m$ such that $0<f(y')<\alpha_3 s$ and $-\alpha_3 s < f(y'') < 0$. In other words, with probability $1-\delta_3$, in each possible $Y_m$, there exist two points $y'$ and $y''$ such that their mapped position $f(y')$ and $f(y'')$ fall on the two sides of $0$, and their distance to 0 is of the order of the standard distance $s$.
\end{lemma}

\begin{proof}
After fixing $Y_m$, let $0\neq y' \in Y_m$ be the point whose mapped position is closest to 0:
\begin{align*}
    y' = \argmin_{0\neq y\in Y_m} d(f(y),0).
\end{align*}
Without loss of generality, assume $f(y')$ is on the positive, right of 0. By the previous Lemmas \ref{not too close lemma}, \ref{one side lemma}, we have with probability $1-\delta_2$ (jointly for all $Y_m$) that
\begin{align*}
    \alpha_1 s < d(f(y'),0) <\alpha_2 s.
\end{align*}
Assume $y' = (y'_1,\ldots,y'_k)$. Our next step is to prove with probability $1-\delta_3'$, with appropriate $\delta_3'<\delta_3$, there exist $0<\eps<1$ such that the absolute value of each coordinate of $y'$ is larger than $\eps m$, i.e. $|y'_i|> \eps m$ for each $i = 1,\ldots,k$. 
To capture these instances, while fixing $b$ of $Y_m$ let us first define the sets when a single $y_i$ is constrained to at most $\eps m$ in absolute value, $Y^{(i,\eps)}_m := Y_{m,\ldots,\eps m,\ldots m}$. Then our set of interest can be expressed using them as follows:
\[
  \tilde Y_{m,\eps} := Y_m \setminus \bigcup_{i=1}^k Y_m^{(i,\eps)} = \{y \mid \eps m < b_iy_i \le m \}.
\]

We will now show that with probability bounded away from 0 the $y'$ obtained is in $\tilde Y_{m,\eps}$. Clearly this means it does not get removed by any $Y_m^{(i,\eps)}$.

Lemma \ref{one side lemma} provides an example in $Y_m$, Corollary \ref{not too close lemma addon} provides some global condition for $Y_m^{(i,\eps)}$, we want to make sure the example does not get removed.
More precisely, assuming that the property of Corollary \ref{not too close lemma addon} holds we have for all $y\in Y_m^{(i,\eps)}$
\begin{equation}
    \label{eq:yepsbound}
    d(f_l(y),0)\geq \alpha_1 s/\eps.
\end{equation}
However, observe that if we had $\alpha_1/\eps>\alpha_2$ the two conditions cannot hold for the same $y$, consequently $y'\notin Y_m^{i,\eps}$. This simply needs tuning $\eps$ so that $\eps<\alpha_1/\alpha_2$. Repeating for $i=1,\ldots,k$ we get $y'\in \tilde Y_{m,\eps}$.

Consider now the set $Y_{\eps m}$ with the same $b$ as with the $Y_m$ we fixed.
Similarly to as in the full $Y_m$, let $0\neq y'_{\eps} \in Y_{\eps m}$ be the point whose mapped position is closest to 0:
\begin{align*}
    y'_{\eps} = \argmin_{0\neq y\in Y_{\eps m}} d(f(y),0).
\end{align*}

To make the logic clear, we have bounded number of favorable events -- on $Y_m,Y_m^{(i,\eps)},Y_{\eps m}$, as applying the previous lemmas -- where each can be tuned to have probability arbitrarily close to 1, so that a union bound on failures can still guarantee we get within the prescribed $\delta_3$ range for the exception probability.

Let us now investigate the properties of $y'_{\eps}$. First, we have three immediate observations based on comparing the ranges of minimization:
\begin{align*}
    y'_{\eps} &\neq y',\\
    d(f(y'),0) &< d(f(y'_{\eps}),0),\\
    y'-y'_{\eps} &\in Y_m.
\end{align*}
Moreover, by Lemma \ref{one side lemma}, $d(f(y'_{\eps}),0)$ is bounded above as
\begin{align*}
    d(f(y'_{\eps}),0) < \frac{\alpha_2 s}{\eps^k}.
\end{align*}
If $f(y'_{\eps})$ is on the negative, left side of $0$, we can assign $y'' = y'_{\eps}$ and the proof is ready, recalling the assumption on $f_l(y')$ being on the other side. If $f(y'_{\eps})$ is on the positive, right side of $0$, we can assign $y'' = y'-y'_{\eps}$ and then we have
\begin{align*}
    f(y'') &= f(y'-y'_{\eps})    = f(y') - f(y'_{\eps})\\
    &= d(f(y'),0) - d(f(y'_{\eps}),0) < 0,
\end{align*}
Confirming it being on the correct side. Moreover, its distance to 0 is also bounded by
\begin{align*}
    d(f(y''),0) < d(f(y'_{\eps}),0) =  \frac{\alpha_2 s}{\eps^k}.
\end{align*}
Noting that we introduce only a constant increase in the exception probability and the distance bounds, i.e. not depending on $n,m$, we conclude.
\end{proof}

At last, we can have our final most refined lemma about mapping evenly on the cycle.
\begin{lemma}
For any $ 0<\delta_4<1$ as exception probability, there exists $\alpha_4>0$ such that with probability at least $1-\delta_4$, for all $Y_m \in \cY_m$ and for all $y \in Y_m$, there exist another two points $y',y''\in Y_m$ such that $0< f(y)-f(y')< \alpha_4 s$ and $0< f(y'') - f(y) < \alpha_4 s$. In other word, for all possible $Y_m$ and any $y\in Y_m$, there exist two "neighbors" $y'$ and $y''$ such that their mapped positions $f(y')$ and $f(y'')$ fall on the two sides of $f(y)$, and their distance to $f(y)$ is of the order of the standard distance $s$.
\end{lemma}
\begin{proof}
The hard work has already been done, let us stitch together the elements.
Consider any point $y\in Y_m$. Again, the Figure \ref{fig:y_inYm} shows the example of $k=3$ and $b = (1,-1,1)$.
\begin{figure}[H]
  \centering
  \includegraphics[scale=0.15]{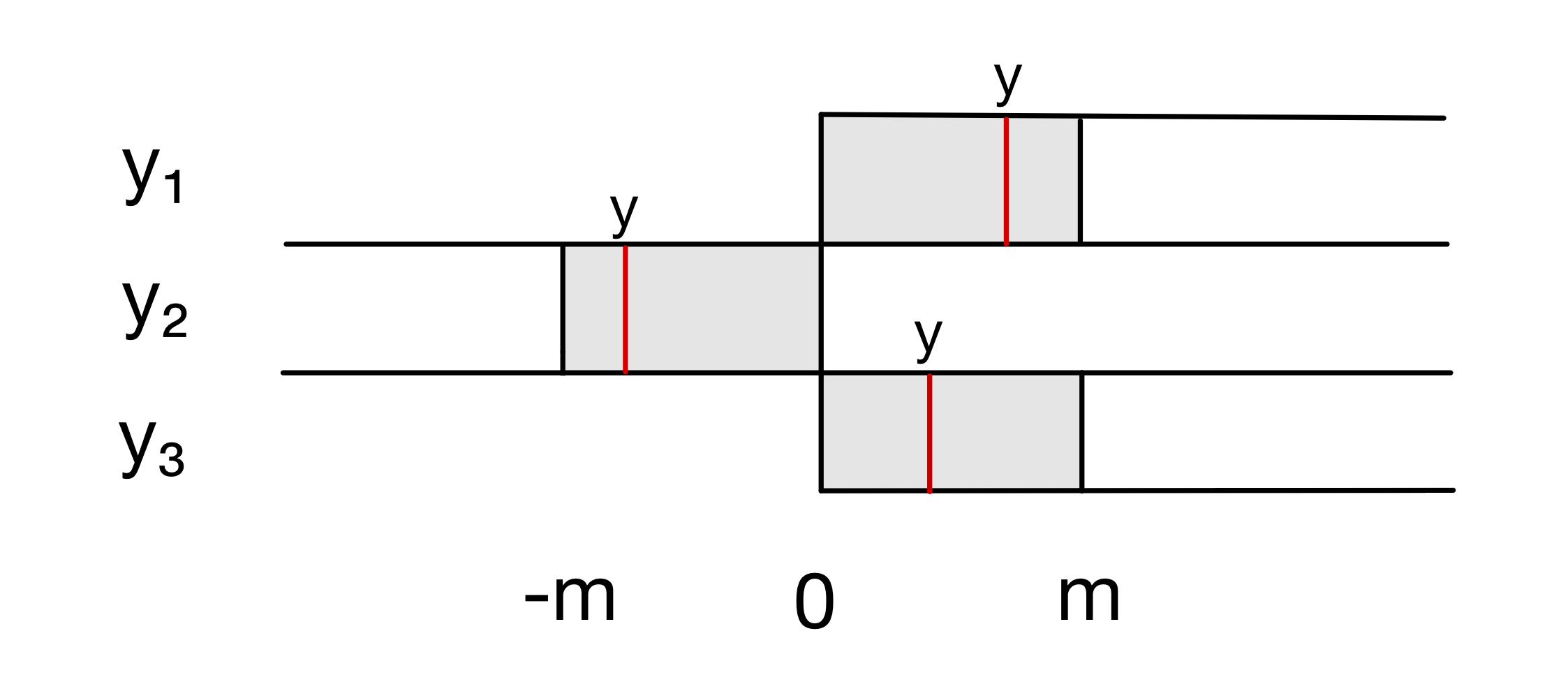}
  \caption{Illustration of $y$ in $Y_m$.} 
  \label{fig:y_inYm}
\end{figure}
Note that we can always choose a subset $Y_{m/2}(y)=y+Y_{m/2}^{b'}$ of $Y_m$ such that it is a shifted $Y_{m/2}$ with $0$ shifted to $y$, with possibly changing the implicit sign pattern $b$ to some different $b'$.
See the example continued in Figure \ref{fig:Ym2_Ym}, where we observe $b'=(-1,1,1)$.

\begin{figure}[H]
  \centering
  \includegraphics[scale=0.15]{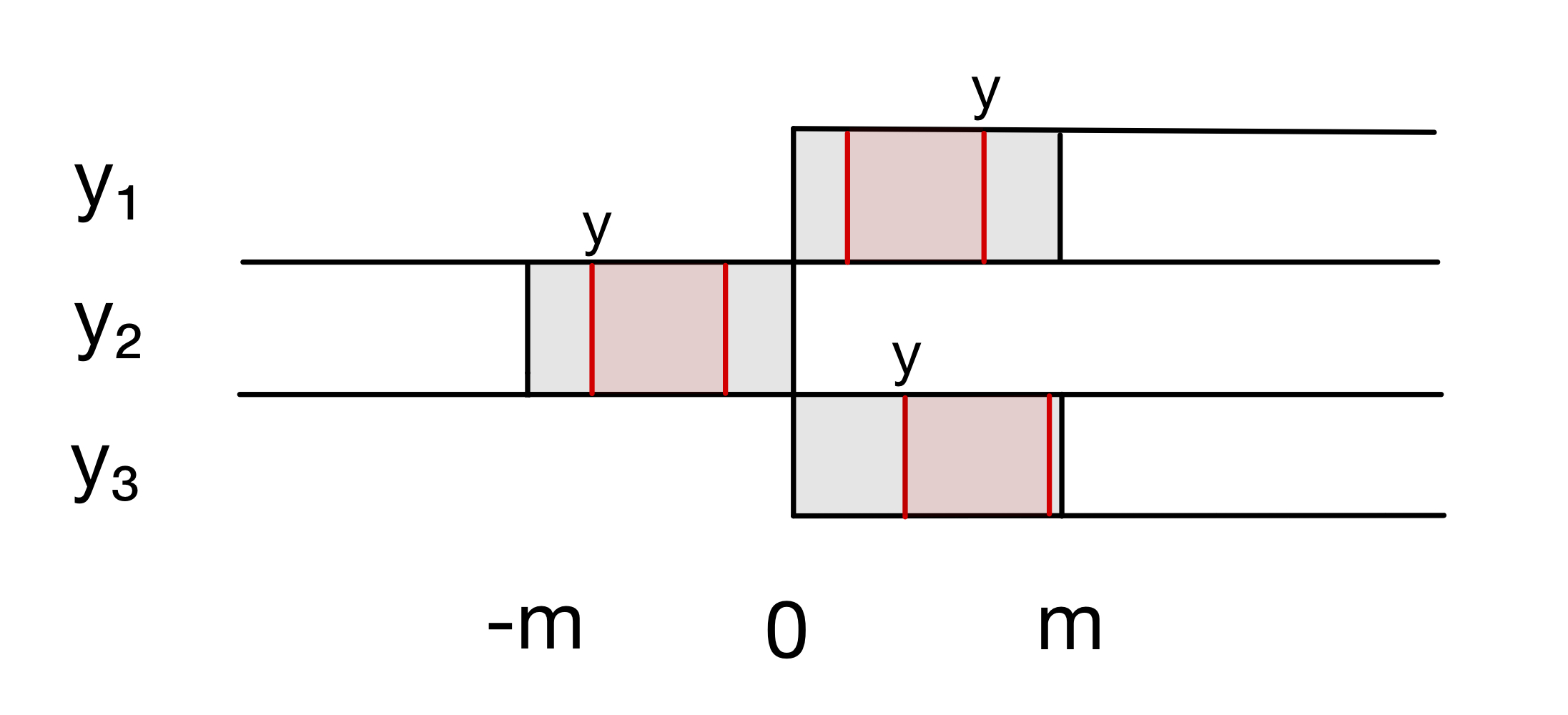}
  \caption{The red shaded area indicates $Y_{m/2}(y)$.} 
  \label{fig:Ym2_Ym}
\end{figure}
By Lemma \ref{both sides lemma} above, there exists two point $Y_{m/2}(y)$ whose mapped positions fall on the two sides of $f(y)$ with distance to $f(y)$ in the order of standard distance $s$, concluding the proof. %
\end{proof}

If the stated property holds, there cannot be a large gap among the mapped points, which would necessarily have a last point before the gap, causing a contradiction.

Let us recast this final statement relating $\mZ_n$ to $\Xi_L$ which we will rely on later, using the value of $m$ set by the definition of $Y$ in Lemma \ref{lm:ytargetprob} and the resulting $s$ by equation $\eqref{eq:standarddistance}$.
\begin{corollary}
\label{cor:hubspread}
For any $m = \lambda_3\sqrt{\rho}\nkd$ and any $0<\delta_4<1$ as exception probability, there exists $\alpha_4>0$ such that with probability at least $1-\delta_4$, for all points $x \in \mZ_n$, there exist $x', x'' \in \Xi_m$ such that $0<x-x'<\alpha_4 \rho^{-\frac k2}\nkb$ and $0<x''-x<\alpha_4 \rho^{-\frac k2}\nkb$.
\end{corollary}

\section{Diffusive effect}
\label{sec:diffuse}

Relying on the findings of the last two sections, we now turn our attention to $X_T$ for appropriate $T$, exploiting the additional randomness compared to $\xi_L$. Recall the target distance $L=\rho \nka$ and the fluctuation $m = \lambda_3\sqrt{\rho}\nkd$ we had for $y$.
Define the set of interest $W$ where the chain will be shown to be spread out, one near reference points collected in $V$:
\begin{align*}
V &= \{v\in \Xi_m \mid |v-h_j|\ge \alpha_5 \sqrt{\rho}n, j=1,\ldots, 2k\},\\
W &= \{w \mid \exists v\in V, |w-v|\le \alpha_4 \sqrt{\rho}\nkb\}, 
\end{align*}
with sufficiently small constant $\alpha_5>0$ to be chosen later. Observe that for $\rho=1$ in favorable cases -- i.e.\ when the property given by Corollary \ref{cor:hubspread} is present -- $W$ recovers most of the cycle, only missing intervals near the hub which gaps can be tuned small by choosing $\alpha_5$.

Consider a fixed $X_0=\xi_0$. For any arc usage $y$ we know where the track should be at distance $L$ following \eqref{eq:xld_def}, now we want to incorporate the time needed for the Markov chain to arrive there.

Let us follow the time evolution of edge and arc usages, let these be $(y_{i,t}),(u_{j,t})$, respectively.
Given target values of $(\hat y_i)$, they define a target position $\xi_L$ and the collection of target arc usages $(\hat u_j)$ (including some $\hat u_j^\pm$) as explained before in \eqref{eq:xld_def} and Lemma \ref{lm:ys_to_us} thus we can set the hitting time of this location as follows:
\begin{equation*}
\omega_L(\hat y) = \omega_L := \min\{t \mid u_{j,t} = \hat u_j, j=1,\ldots, 2k^*, X_t = \xi_L\}.
\end{equation*}
In effect, outside of the event $\{y=\hat y\}$ we get $\infty$, while on that event it encodes the time needed for accumulating a travel distance of exactly the prescribed $L$ (not counting for long range edge moves) via matching $u_{j,n}$ with $\hat u$.
Similarly, we define hitting of hubs before and after this target distance $L$ in a slightly asymmetric way. For the target position $f(\hat y) \in (h_j,h_{j+1}]$ let
\begin{align*}
    \omega_L^-(\hat y) = \omega_L^- &:=\omega_{L-f(\hat y)+h_j+\log^2n}(\hat y),\\
    \omega_L^+(\hat y) = \omega_L^+ &:=\omega_{L-f(\hat y)+h_{j+1}}(\hat y).
\end{align*}
Informally, $\omega_L^+(\hat y)$ is the hitting time of the hub after $f(\hat y)$, after having travelled distance $L$. 
Similarly, $\omega_L^-(\hat y)$ is the hitting time of the point $\log^2 n$ after the last hub before $f(\hat y)$.

Define $L_h:=L-3k\rho\nkc \log^2n$, representing the drift ensured to happen away from the hubs. The following decomposition lemma will help us separate complicated parts and highlight diffusive components of these hitting times.

\begin{lemma}
\label{lm:diffusedecompose}
Assume $\cB_1$.
There exists a constant probability $\mu_1>0$ with the following. For any fixed $X_0=\xi_0$ there is an event $\cB_{X_0}$ with $\mP(\cB_{X_0} \mid \cB_2, y)\ge \mu_1$, and
for any $\hat y \in Y = [-\lambda_3\sqrt{\rho}\nkd,\lambda_3\sqrt{\rho}\nkd]^k$ conditionally on $\cB_2\cap \cB_{X_0}\cap\{y=\hat y\}$ the overall hitting time $\omega_L(\hat y)$ is equidistributed with
\[
\sum_{i=1}^{L_h} \tau_1^i+ Z_{n,\hat y},
\]
which is a sum of independent variables: 
\begin{itemize}[itemsep=0pt, topsep=0pt]
    \item $\tau^i_1$ is an i.i.d.\ series, each is a hitting time of $+1$ from $0$ conditioned to not backtrack to $-\log^2n$ (on $\mZ$), in the usual biased setup;
    \item $Z_{n,\hat y}$ is a non-negative variable with $\mE(Z_{n,\hat y})=O(\rho\nkc \log^2n)$.
\end{itemize}

Let us express the position given by $\hat y$ as $f(\hat y) = h_j+s\in V$, where $s$ stands for the location within the arc. Conditioned on the event $\cB_2\cap\cB_{X_0}\cap\{y=\hat y\}$ the overall hitting time $\omega_L^-$, for a bit ahead of the last hub visited on the track, is equidistributed with an expression similar as above:
\[
\sum_{i=1}^{L_h-s+\log^2n} \tau_1^i+ Z_{n,\hat y}.
\]
\end{lemma}

Observe that the distribution of the correcting variable $Z_{n,\hat y}$ might depend on $\hat y$.%
\begin{proof}
    The two representation claims are nearly the same, only the length of the track changes. We spell out the proof w.r.t.\ $\omega_L$. The overall picture until reaching distance $L$ is an interlacing of walk pieces wandering near the hubs - depending on the $\hat y$ to be hit -- and long drifts hitting the next hub ahead on the cycle. The former will provide the corrective $Z_{n,\hat y}$ while the latter will bring the large i.i.d.\ composition. Let us expand this formally.

    If starting from $X_0 \in [h_j, h_j+\log^2 n)$, then we might have an unlucky start moving backwards, in the negative direction, the favorable case is $\cB_{X_0} = \{\tau^*_{h_j}>\tau^*_{h_j+\log^2n}\}$. Here $\tau^*_v$ denotes the hitting time of a given point $v$ on the cycle, in contrast to $\tau^i_1$, etc.\ capturing hitting times for relative movements. For the probability of this event we have
    $$
    \mP(\cB_{X_0}\mid \cB_2,y)\geq \mP(\tau_{-1}>\tau_{\log^2n})=\frac p q + o(1).
    $$
    Let us emphasize that the $o(1)$ term does not depend on $X_0,y$, thus we get the event with the uniform bound on the conditional probability as postulated.
    From now on, we will condition on $\cB_{X_0}\cap \cB_2$. 

    With the $\hat y$ given, choose also a collection $\hat r\in R$ as in \eqref{eq:rrange} and moreover even a track $\hat \xi$ to be followed. Once the claim is shown in this context, it immediately follows for their mixture once $\hat r, \hat \xi$ are set free.

    Let us now follow the Markov chain while conditioning on $\{\xi=\hat \xi\}$ (besides $\cB_2$) to get a view on the hitting time of the end of it. There will be two type of components:
    \begin{itemize}
        \item From a hub we reach $\log^2 n$ ahead on the arc next to it or on the one across. The time needed at each such occasion is independent, let us combine the steps needed to get $\tilde Z_{n,\hat\xi}$.
        \item From there on until the next hub, we take a few steps to hit $x+1$ from $x$, these form an i.i.d.\ collection of $\tau_1$ distributed variables, relying on $\cB_2$.
    \end{itemize}
For the number of $\tau^i_1$-s to be added, there is a total length $L$ travelled but $\log^2n$ omitted at each visit of a hub after traversing an arc, in total at $\sum_{j=1}^{2k^*} u_j$ occasions. Using Lemma \ref{lm:ys_to_us}, there are at most $3k\rho\nkc$ such occurrences for $n$ large enough, therefore at least $L-3k\rho\nkc\log^2n$ travel remains to be covered with independent $\tau^i_1$ steps, matching the first term in the expression claimed.

If any $\tau^i_1$ term remains, it is added to $\tilde Z_{n,\xi}$ to form the final $Z_{n,\hat\xi}$. As all terms of the decomposition have been included, the distribution representation claim is immediate. 

Concerning the distribution of $Z_{n,\hat\xi}$ - and afterwards the mixture $Z_{n,\hat y}$ - it is clearly non-negative by construction. For the bound on expectation, we rely on $\mE(\tau_1) = 1/(p-q)+O(\exp(-\lambda_2\log^2 n))$, a standard fact from the asymmetric gambler's ruin problem.

Combining with the maximal number $L-L_h$ of $\tau_1$ terms possibly included in $Z_{n,\hat\xi}$ their contribution is at most $\frac{3k\rho}{p-q}\nkc\log^2n+O(\exp(-\lambda'_2\log^2 n))$ in expectation. For the parts escaping the hubs, for each occasion we rely on $\mE(\tau_{\log^2n})=\log^2n/(p-q)+O(\exp(-\lambda'_2\log^2 n))$ as before, but we have to include the condition of escaping on the arc prescribed by $\hat \xi$. 
Conditioning on the eventual decision at a hub we see 
\begin{equation}
    \label{eq:taulog2n}
\mE(\tau_{\log^2n})=\mE(\tau_{\log^2n}\mid G)(p_G+o(1))+\mE(\tau_{\log^2n}\mid J)(p_J+o(1)),
\end{equation}
the $o(1)$ corrections due to the conditioning on $\cB_2$, similarly to how hub-dependent $p_G^{(j)}$ were adjusted in Lemma \eqref{lm:pg_pj}. Indeed, for a decision $G$ (or $J$) we see
\begin{equation*}
    \mP(G\mid \cB_2) = \frac{\mP(G\cap\cB_2)}{\mP(\cB_2)} = \frac{\mP(G)+O(\exp(-\lambda_2\log^2 n))}{1+O(\exp(-\lambda_2\log^2 n))}= \mP(G)(1+O(\exp(-\lambda_2\log^2 n)).
    \end{equation*}
Rearranging \eqref{eq:taulog2n} while using Lemma \ref{lm:pg_pj} leads to the bounds
\begin{align*}
    \mE(\tau_{\log^2n}\mid G) &\le \frac{\mE(\tau_{\log^2n})}{p_G+o(1)} \le \frac{p-q+2a+o(1)}{(p-q+a)(p-q)}\log^2 n, \\
    \mE(\tau_{\log^2n}\mid J) &\le \frac{\mE(\tau_{\log^2n})}{p_J+o(1)} \le \frac{p-q+2a+o(1)}{a(p-q)}\log^2 n.  
\end{align*}
Once again we have at most $3k\rho\nkc$ terms of those types included in $\tilde Z_{n,\hat\xi}$, combining with the above $\tau_1$ contribution we get
\[
\mE(Z_{n,\hat\xi}) \le (L-L_h)\cdot\mE(\tau_1)+3k\rho\nkc \cdot O(1)\log^2n =O(\rho\nkc \log^2n),
\]
confirming the bound on the expectation and immediately for the mixture $Z_{n,\hat y}$.
\end{proof}

From now on let $T = \frac{\rho}{p-q} \nka$ be the fixed target time to analyze, via the stopping times above. Promising diffusion near the hubs can be seen as follows.
\begin{lemma}
\label{lm:xspread1}
Assume $\cB_1$. For any $\alpha_6>0$ there exists $\mu_2>0$ so that for $n$ large enough we have the following. 
For any $\hat y \in Y$, if $f(\hat y) \in V$, then around this target,
\begin{equation}
    \mP(X_T=f(\hat y)+s' \mid y=\hat y)\geq \frac{\mu_2}{\sqrt{\rho}}\nkmb \qquad \forall s'\in [-\alpha_6\sqrt{\rho}\nkb,\alpha_6\sqrt{\rho}\nkb].
\end{equation}
\end{lemma}
\begin{proof}
First let us note that the conditioning on $\cB_2$ is harmless to play with. For $\hat y\in Y$ relying on Lemma \ref{lm:ytargetprob} we get the lower and upper bounds
\begin{equation}
\label{eq:B2fine}
    1\geq\mP(\cB_2\mid y=\hat y) = \frac{\mP(\cB_2\cap \{y=\hat y\})}{\mP(y=\hat y)}\geq \frac{\mP(y=\hat y)-\exp(-\lambda_2\log^2 n)}{\mP(y=\hat y)}\geq 1-\exp(-\lambda_2'\log^2 n).
    \end{equation}
Fixing $\hat y\in Y$, we express the position within the arc as $f(\hat y) = h_j+s$ as before and use the conditioning and the decomposition of Lemma \ref{lm:diffusedecompose}, with various intermediate constants $\beta_i, \beta_i'>0$ to come.
We may bound $Z_{n,\hat y}$ by Markov's inequality as
\[
\mP(Z_{n,\hat y}\geq\beta_1\rho\nkc\log^3 n) \leq \frac{\beta_1'}{\log n}.
\]
We can then combine this with a (Local) Central Limit approximation of the sum of the $\tau^i_1$ terms, with overall mean $\frac{\rho}{p-q}\nka-\frac{s}{p-q}+O(\exp(-\lambda'_2\log^2 n))$ and overall variance $\Theta(\rho\nka)\gg \rho^2 n^{\frac2{k+1}}\log^6 n$, thus dominating the contribution of $Z_{n,\hat y}$. A crude minorizing estimate is possible uniformly as follows: %
\begin{equation}
    \begin{aligned}
        \label{eq:tauLminus-spread}
\mP(\omega^-_L=T-t \mid y=\hat y) &\geq \mP(\omega^-_L=T-t \mid \cB_2,\cB_{X_0},y=\hat y) \mP(\cB_{X_0} \mid \cB_2, y=\hat y) \mP(\cB_2, \mid y=\hat y)\\ &\geq\frac{\beta_2'}{\sqrt{\rho}}\nkmb 
\qquad \forall t\in \left[\frac{s}{p-q}-\beta_2\sqrt{\rho}\nkb,\frac{s}{p-q}+\beta_2\sqrt{\rho}\nkb\right],
    \end{aligned}
\end{equation}
where we include the condition on $\cB_{X_0}$ and the above comment \eqref{eq:B2fine} on $\cB_2$. Let us emphasize that this overall bound does not need the conditioning on $\cB_2\cap\cB_{X_0}$ anymore. Here $\beta_2$ can be arbitrarily increased at the price of decreasing $\beta'_2$ (requiring large enough $n$).

In essence this is a bound on the random hitting time of the reference point set before $f(\hat y)$, now we estimate in a similar way the remaining $t$ steps to get a view on the position at a deterministic time. Again the Local Central Limit minorized by a uniform shows for any $t$ as in \eqref{eq:tauLminus-spread} considered,
\begin{equation}
\label{eq:XTspread}
\mP(X_T = h_j+\log^2n+(p-q)t+s'' \mid \omega^-_L=T-t, y=\hat y) \geq \frac{\beta_3'}{\sqrt{s}}, \qquad \forall s''\in[-\beta_3\sqrt{s},\beta_3\sqrt{s}],    
\end{equation}
where we used that $s=\Omega(\sqrt{\rho}n)\gg \beta_2\sqrt{\rho}\nkb$ from the definition of $V$. 

Scanning and merging through the events $\{\omega_L^-=T-t\}$ can be used to get a minorizing measure for the position of $X_T$ conditioned on $\{y=\hat y\}$:
\begin{equation}
    \begin{aligned}
\label{eq:XTspread-comb}
    \mP(X_T = f(\hat y) + s' \mid y=\hat y) &\geq \sum_t \mP(X_T = f(\hat y) + s' \mid \omega^-_L=T-t, y=\hat y) \mP(\omega^-_L=T-t \mid y=\hat y).
\end{aligned}
\end{equation}
In order to utilize inequalities \eqref{eq:tauLminus-spread} and \eqref{eq:XTspread}, we have to identify the range of $t$ in the summation when both estimates are in force. For a fixed target $f(\hat y)+s'=h_j+s+s'$ the constraints of \eqref{eq:XTspread} can be rewritten in terms of $t$ as
\begin{align*}
    |s''| = |s+s'-\log^2n-(p-q)t| &\le \beta_3 \sqrt{s},\\
    \left|t-\left(\frac{s+s'}{p-q} - \frac{\log^2n}{p-q}\right)\right| &\le \beta_3 \frac{\sqrt{s}}{p-q},
\end{align*}
thus thanks to the assumption on $s'$ forms an interval around $s/(p-q)+O(\sqrt{\rho}n^\frac{k+2}{2k+2})$ of width $O(\sqrt{n})$. In \eqref{eq:tauLminus-spread} we see a similar but wider interval, around $s/(p-q)$ of width $O(\sqrt{\rho}n^\frac{k+2}{2k+2})$. In the end, if $\alpha_6 < \beta_2$ their intersection is still of length $2\beta_3\sqrt{s}=O(\sqrt{n})$, which determines the simultaneous validity of the bounds. Including the possible $t$ in \eqref{eq:XTspread-comb} we arrive at
\begin{equation}
\label{eq:XTspread-overall}
\mP(X_T = f(\hat y) + s' \mid y=\hat y) \geq 2\beta_3\sqrt{s} \cdot \frac{\beta_2'}{\sqrt{\rho}}\nkmb \cdot \frac{\beta_3'}{\sqrt{s}} = \frac{2\beta_3\beta'_3\beta'_2}{\sqrt{\rho}}\nkmb.
\end{equation}

To complete this part of the analysis, we have to ensure exceptional events do not change the behavior. On one hand, the time splitting for the two-phase analysis is legit, the event $\{\omega_L^- <T\}$ has probability bounded away from 0 as clear from \eqref{eq:tauLminus-spread} noting that $s\gg \sqrt{\rho}\nkb$ on $V$. 
On the other hand, afterwards the issue could be hitting a hub within the time-range $[\omega_L^-,T]$ analyzed. Given that the starting point $h_j+\log^2n$ at $\omega_L^-$ and the endpoint studied in $f(\hat y)\in V$ at $T$ are far from hubs by construction, if the walk had visited a hub during the walk, it would imply a substantial backtrack of length at least $\log^2 n$ to occur which has probability $O(\exp(-\lambda_2\log^2n))$ %
similarly to as in Lemma \ref{lm:preferableevents}.

To conclude, choose any $\alpha_6<\beta_2$ to ensure a range of validity for the bounds jointly, then any $\mu_2<2\beta'_2\beta_3\beta_3'$ and the minorization will hold for large enough $n$ even with the exceptional events taken into account.
\end{proof}

\begin{lemma}
\label{lm:xspread2}
Assume $\cB_1$. For any $\alpha_4>0$ there exists $\mu_3>0$ so that for $n$ large enough for any $w\in W$ we have
\begin{equation*}
\mP(X_T = w) \ge \rho^{-\frac{k+1}{2}}\frac{\mu_3}{n}.
\end{equation*}
Moreover, for any $\hat y\in Y, w\in W$ so that $f(\hat y)\in V, |w-f(\hat y)|\leq \alpha_4 \sqrt{\rho}\nkb$ we similarly have
\begin{equation*}
\mP(X_T = w, y=\hat y) \ge \rho^{-\frac{k+1}{2}}\frac{\mu_3}{n}.
\end{equation*}
\end{lemma}

\begin{proof}
Note that the first claim is only a slightly weaker, simplified form of the second one. The only thing we need to do is to combine the condition and the conditional probability investigated before. If $w\in W$, it means there is $\hat y\in Y$ such that $f(\hat y)\in V$ and $|f(\hat y)-w|\leq \alpha_4 \sqrt{\rho}\nkb$.
    By Lemma \ref{lm:ytargetprob} for the probability of having encountered this $\hat y$ we have that 
    \[
    \mP(y=\hat y) \geq \lambda_6 \rho^{-\frac k 2}n^{-\frac{k}{2k+2}}.
    \]
    Then by Lemma \ref{lm:xspread1}, with $\alpha_6=\alpha_4$ 
    there is $\mu_2>0$ so that
    \[
    \mP(X_T=w \mid y=\hat y) \geq \mu_2 \rho^{-\frac 12}\nkmb.
    \]
Combining the two we conclude.
\end{proof}

Note that until now we did not make use of the spread of points in $V$ and have claims that are true for general $0<\rho\le 1$. However, to confirm mixing, the focus is crucial.

\begin{proof}[Proof of Theorem \ref{eq:mainthm_bound}]
Set $\rho=1$ and set $X_0$ free. Notice that the exceptional event in 
Corollary \ref{cor:hubspread} is about the edge placement and not the walker or its starting point, thus we can analyze the chain jointly from all starting points for favorable cases. 

In this case $W$ is mostly covering $\mZ_n$, large parts of the arcs except near the hubs. For $\alpha_5$ small enough $|W|\geq n/2$ once $n$ is sufficiently large by the definition of $V,W$, then for any two initialization, Lemma \ref{lm:xspread2} provides
\[
\|P^T(x,\cdot)-P^T(y,\cdot)\|_{\rm TV} \leq 1-\frac{\mu_3}{2},
\]
where taking supremum in $x,y$ on the left hand side the mixing quantity $\bar d(T)$ is obtained, immediately resulting in $\bar d(T) \leq 1-\frac{\mu_3}{2}$. The submultiplicativity of $\bar d(\cdot)$ \cite[Chapter 4.4]{levin2017markov} ensures less than a $\frac{2}{\mu_3}\log \frac 1 \eps$ factor is needed to reach below $\eps$, together with $d(\cdot)\leq \bar d(\cdot)$ confirming the upper bound on mixing. 

Turning to the lower bound we will not rely on the results of Section \ref{sec:spread}, only implicitly require $\cB_1$, resulting in an a.a.s.\ statement. We will work by tuning $\rho$. Choose $\rho$ small enough so that $\rho^{\frac{k+1}2}<\frac {\mu_3}4$ and let $\alpha_5$ sufficiently small as above. By the shift-invariance of $\xi_L$ for appropriate $X_0=\xi_0=x$ we have 
$|Y\cap f^{-1}(V)|\geq\frac 12 |Y| = 2^{k-1} \lambda_3^k\rho^{\frac k2}n^{\frac{k}{2k+2}}$. 
Recall that each such $\hat y$ might lead to a final location $w$ around $f(\hat y)$ with range $2\alpha_4 \sqrt{\rho}\nkb$,
using Lemma \ref{lm:xspread2} and combining we get
\[
P(X_T\in W)\geq 2^k\alpha_4\lambda_3^k\mu_3,
\]
which is independent of the choice of $\rho$.
At the same time, $|W|\leq 2^{k+1}\alpha_4\lambda_3^k\rho^{\frac{k+1}2} n$, scaling with $\rho$. Thanks to the condition on $\rho$, using this set for comparing $X_T$ with the stationary uniform $\pi$ leads to
\[
\|P^T(x,\cdot)-\pi(\cdot)\|_{\rm TV} \geq P^T(x,W)-\pi(W) \geq 2^{k-1} \alpha_4\lambda_3^k\mu_3,
\]
confirming the lower bound claimed, with $\gamma^*<\rho/(p-q)$ and $\eps^*$ as in the equation above.

Finally, in order to check for cutoff, we combine the upper and lower bounds. 
Note that the claim on no cutoff does not depend on the exception probability, it is sufficient to confirm along a subsequence of the series of Markov chains in question.

Use the same small $\rho$ range as above, resulting in an estimate at least $4/n$ in Lemma \ref{lm:xspread2}.
Tuning $\alpha_5$ small enough, on the favorable event (w.r.t.\ Corollary \ref{cor:hubspread}) 
we have $|W| \geq \beta \rho^{\frac{k+1}2} n$ for some constant $\beta>0$.
Similarly as before, we have the upper bound on the total variation distance from the (uniform) stationary distribution $\pi$ for any starting point
\[
\|P^T(x,\cdot)-\pi(\cdot)\|_{\rm TV} \leq 1-\frac 1n |W| \leq 1-\beta \rho^{\frac{k+1}2},
\]
providing an upper bound on $d(T)$ still with $T = \frac{\rho}{p-q} \nka$, now with the small $\rho$ chosen.

Although these lower and upper bounds are very far, they are bounded away from 0 and 1 at a time of the appropriate order. In the end, choose $\rho_1\neq\rho_2$ small enough in the regime above. By construction the respective times $T$ will have constant ratio, while the total variation distance will be bounded away from 0 and 1 for both, confirming non-cutoff behavior.
\end{proof}

Let us note that the dependence on the exception probability $\delta$ in $\gamma(\delta)$ for the upper bound is through the dependence on $\mu_3$, itself coming from
Lemma \ref{lm:xspread2}, where the $\alpha_4$ used depends on $\delta$ in Corollary \ref{cor:hubspread}.

\section{Numerical results}
\label{sec:simulation}

Let us now complement our analytic results by numerical demonstrations.
First, we assess the effect of the length of the added edges.
\begin{figure}[H]
  \centering
    \resizebox{0.50\textwidth}{!}{\import{images/}{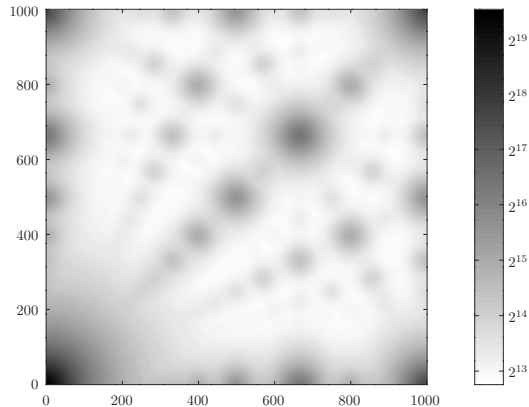}}
  \caption{$n=2000$, 2 added edges. The $x$-axis and $y$-axis represents the length of the 2 edges. The color represents $t_{\rm mix}(\frac{1}{4})$.}
  \label{2000-2}
\end{figure}
In Figure \ref{2000-2}, we work with a cycle with $n=2000$ points. We place 2 extra edges ($k=2$) to the cycle and iterate through all possible length of the 2 edges. The $x$-axis and $y$-axis represents the length of the 2 edges. For each different pair of edge lengths, we pick two position of the two edges randomly. 
Consistently with our discussion before, we analyze the Markov chains with the two added edges, computing the mixing time numerically. A heatmap of the resulting values is shown in Figure \ref{2000-2}, on a logarithmic scale.
The color in Figure \ref{2000-2} represents $t_{\rm mix}(\frac{1}{4})$. Theorem \ref{thm:mainthm} indicate that with 2 random edges, the mixing time is approximately
\begin{align*}
    2000^{\frac{3}{4}} \approx 2^{14.6},
\end{align*}
and the figure is in line with this approximation. Note that there are black dots and patches is Figure \ref{2000-2}, which indicate exceptional large mixing time. There is a rather regular pattern visible, the patches appear when the edge lengths are near (compatible) fractions of $n$ with small denominator.
We can observe that the mixing time is very sensitive to the edge lengths chosen.

Let us now turn our focus to the positioning of the edges which was previously randomized. Therefore, to be able to present, we randomly pick certain edge lengths and work with those.
\begin{figure}[H]
  \centering
  \begin{subfigure}[b]{0.48\textwidth}
      \includegraphics[width=\textwidth]{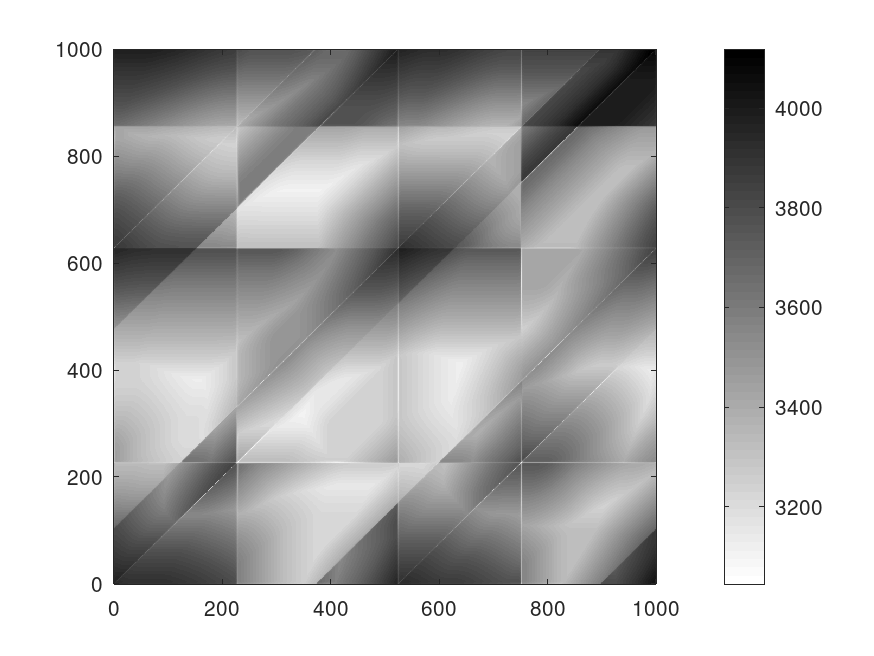}
        \caption{$n=1000$, 3 edges with lengths $(227, 372, 476)$.}
  \label{1000-3}
  \end{subfigure}
  \hfill
  \begin{subfigure}[b]{0.48\textwidth}
      \includegraphics[width=\textwidth]{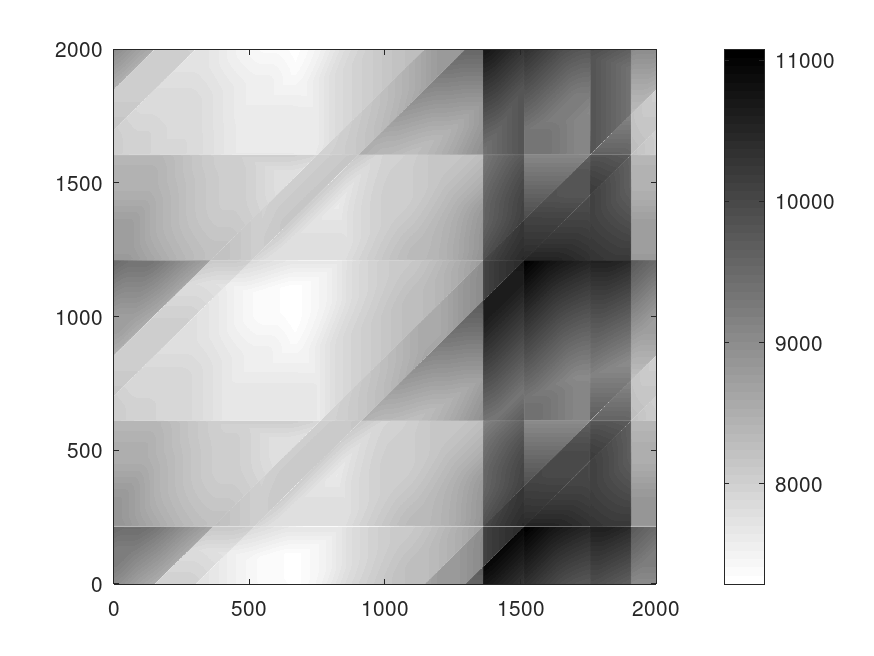}
    \caption{$n=2000$, 3 edges with lengths $(395, 995, 151)$.}
  \label{2000-3}
  \end{subfigure}
  \caption{Fluctuations depending on edge positions. The location of the left endpoint of the first edge is 0. The $x$-axis and $y$-axis represents the location of the left endpoints of the second and third edge. The color represents $t_{\rm mix}(\frac{1}{4})$.}
\end{figure}
In Figure \ref{1000-3}, we set $n=1000$, $k=3$, and the lengths of the three edges are $(227, 372, 476)$, while in
Figure \ref{2000-3} we work on a larger cycle $n=2000$, edge lengths chosen as $(395, 995, 151)$.
The left, negative endpoints of the three edges are $0,x,y$, allowing a 2D heatmap presentation again. The color thus represents $t_{\rm mix}(\frac{1}{4})$ in the given cycle, now on a linear scale. Our results indicate that the mixing times should be approximately
\begin{align*}
    1000^{\frac{5}{4}} \approx 5623 \qquad \text{and} \qquad 2000^{\frac{5}{4}} \approx 13375
\end{align*}
respectively, and the figures agree with this approximation. Moreover, we can see that the range of the mixing time is narrower than before in Figure \ref{2000-2}, suggesting that the mixing time is less affected by the location of the edges.

Finally, we numerically evaluate whether the exception probability imposed in Theorem \ref{thm:mainthm} is reasonable or an a.a.s.\ bound of the same order should be hoped hoped for. To this end, we fix $k=2$ and for certain $n$ we scan all possible pairs of edge lengths, in each case with a random placement and evaluate $t_{\rm mix}(\frac 14)$. In contrast to Figure \ref{2000-2}, we simply dump and sort the resulting values. After further scaling by $n^{4/3}$, the reference value provided by Theorem \ref{thm:mainthm}, we get Figure \ref{fig:sorteda}.

We both see a striking match of the curves and also the higher values diverging. To get a closer view, we normalize by the curve produced for $n=200$ and plot the relative values in Figure \ref{fig:sortedb}. Observe that the larger $n$ still present relative values out of bounds. This is consistent noting that there exist bad edge setups with $\Theta(n^2)$ mixing time.
\begin{figure}[H]
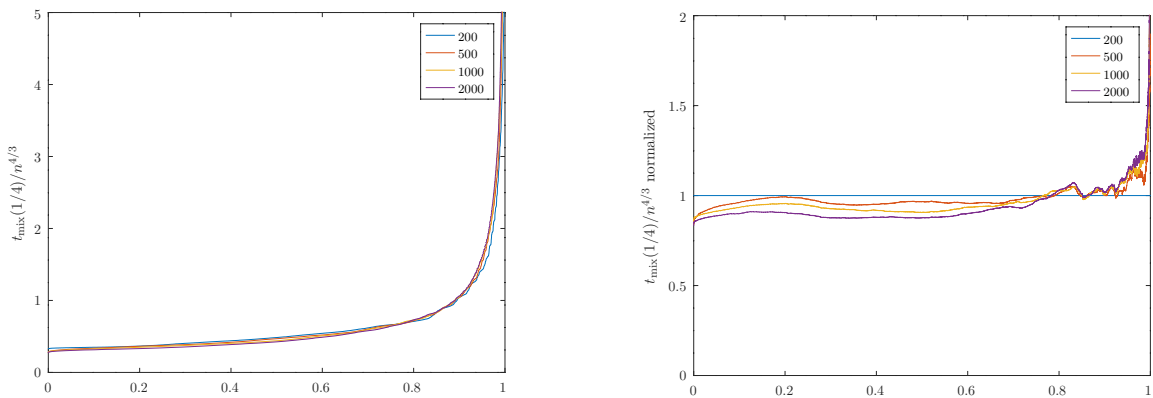

    \centering  
    \begin{subfigure}[b]{0.48\textwidth}
    \centering
    \resizebox{1.00\textwidth}{!}{\input{images/tmix_2edge_sorted.tex}}    
        \caption{Mixing times scaled with the reference bound $n^{4/3}$.}
         \label{fig:sorteda}
    \end{subfigure}
    \hfill
    \begin{subfigure}[b]{0.48\textwidth}
    \centering
    \resizebox{1.00\textwidth}{!}{\input{images/tmix_2edge_sorted_normalized.tex}}
        \caption{Scaled mixing times further normalized by the curve of $n=200$.}
         \label{fig:sortedb}
    \end{subfigure}
    \caption{Sorted mixing times for all possible edge lengths with $k=2$, comparing for few values of $n$.}
    \label{fig:sorted}
\end{figure}
These plots suggest there is indeed no concentration of the mixing times, no a.a.s.\ upper bound of $O(\nka)$ is expected.

\section{Conclusions}
\label{sec:conclusions}

Overall, a near-quadratic mixing speedup has been shown as a result of a minuscule perturbation to the original non-reversible Markov chain on the cycle. There are interesting details on the conditions to mention though.

For the lower bound, if $q=0$ then $\cB_1$ holds automatically, consequently the lower bound holds for \emph{all} choices of extra edges, improving on the a.a.s.\ claim. To get a similar global lower bound for $q>0$, non-trivial backtracking events would have to be handled, which is shifting to the more complex setup of adding small graphs rather than edges to the cycle.

For the upper bound, it would be great to see a range of examples or a matching lower bound which confirm the need of exceptional events (of non-vanishing probability). Edge lengths near fractions of $n$ with small denominator are expected to lead to such examples, as already visible in Figure \ref{2000-2}, likely leading to Diophantine elaborations.

In a broader context, the question arises when can an appropriate small, local perturbation of the Markov chain affect the global mixing behavior so strongly.

\bibliographystyle{siam}
\bibliography{refs}

\end{document}